\documentclass{article}
\usepackage{amssymb}
\usepackage{amsfonts}
\usepackage{url}
\usepackage{amsmath}
\usepackage{graphicx}

\newtheorem{theorem}{Theorem}

\newtheorem{corollary}{Corollary}

\newtheorem{definition}{Definition}

\newtheorem{lemma}{Lemma}

\newtheorem{proposition}{Proposition}
\newtheorem{remark}{Remark}

\newenvironment{proof}[1][Proof]{\textbf{#1.} }{\ \rule{0.5em}{0.5em}}

\begin{document}

\begin{center}
{\LARGE Linearized} {\LARGE Boltzmann Collision Operator: II. Polyatomic
Molecules Modeled by a }

{\LARGE Continuous Internal Energy Variable}\bigskip

{\large Niclas Bernhoff\smallskip }

Department of Mathematics, Karlstad University, 65188 Karlstad, Sweden

niclas.bernhoff@kau.se
\end{center}

\textbf{Abstract:}{\small \ The linearized collision operator of the
Boltzmann equation for single species can be written as a sum of a positive
multiplication operator, the collision frequency, and a compact integral
operator. This classical result has more recently, been extended to
multi-component mixtures and polyatomic single species with the
polyatomicity modeled by a discrete internal energy variable. In this work
we prove compactness of the integral operator for polyatomic single species,
with the polyatomicity modeled by a continuous internal energy variable, and
the number of internal degrees of freedom greater or equal to two. The terms
of the integral operator are shown to be, or be the uniform limit of,
Hilbert-Schmidt integral operators. Self-adjointness of the linearized
collision operator follows. Coercivity of the collision frequency are shown
for hard-sphere like and hard potential with cut-off like models, implying
Fredholmness of the linearized collision operator.}

\textbf{Keywords:}{\small \ Boltzmann equation, Polyatomic gases, Linearized
collision operator, Hilbert-Schmidt integral operator}

\textbf{MSC Classification:}{\small \ 82C40, 35Q20, 35Q70, 76P05, 47G10}

\section{Introduction\label{S1}}

The Boltzmann equation is a fundamental equation of kinetic theory of gases.
Considering deviations of an equilibrium, or Maxwellian, distribution, a
linearized collision operator is obtained. The linearized collision operator
can in a natural way be written as a sum of a positive multiplication
operator, the collision frequency, and an integral operator $-K$. Compact
properties of the integral operator $K$ (for angular cut-off kernels) are
extensively studied for monatomic single species, see e.g. \cite{Gr-63,
Dr-75, Cercignani-88, LS-10}. The integral operator can be written as the
sum of two compact integral operators, in the form of a Hilbert-Schmidt
integral operator and an approximately Hilbert-Schmidt integral operator,
i.e. an operator, which is the uniform limit of Hilbert-Schmidt integral
operators (cf. Lemma $\ref{LGD}$ in Section $\ref{PT1}$) \cite{Glassey}, and
so compactness of the integral operator $K$ follows. More recently,
compactness results were also obtained for monatomic multi-component
mixtures \cite{BGPS-13}, see also \cite{Be-21a} for a different approach,
and for polyatomic single species, where the polyatomicity is modeled by a
discrete internal energy variable \cite{Be-21a}. In this work, we consider
polyatomic single species, where the polyatomicity is modeled by a
continuous internal energy variable \cite{BDLP-94, BBBD-18, GP-20}. We
restrict ourselves to the physical case when the number of internal degrees
of freedom is greater or equal to two. The compactness property in the case
when the molecules is restricted to undergo resonant collisions \cite{BRS-22}%
, i.e. collisions where the sum of the internal energies is conserved under
the collision, is recently considered in \cite{BBS-22}.

Motivated by an approach by Kogan in \cite[Sect. 2.8]{Kogan} for the
monatomic single species case, a probabilistic formulation of the collision
operator is considered as the starting point. With this approach, cf. \cite%
{Be-21a}, it is shown, based on modified arguments in the monatomic case,
that the integral operator $K$ can be written as a sum of a Hilbert-Schmidt
integral operator and an operator, which is the uniform limit of
Hilbert-Schmidt integral operators (and even might be an Hilbert-Schmidt
integral operator itself) - and so compactness of the integral operator $K$
follows. The operator $K$ is self-adjoint, as well as the collision
frequency, why the linearized collision operator as the sum of two
self-adjoint operators of which one is bounded, is also self-adjoint.

For hard sphere like models and hard potential with cut-off like models,
bounds on the collision frequency are obtained. Then the collision frequency
is coercive and becomes a Fredholm operator. The set of Fredholm operators
is closed under addition with compact operators, why also the linearized
collision operator becomes a Fredholm operator by the compactness of the
integral operator $K$. For hard sphere like models, as well as, "super hard"
potential like models, the linearized collision operator satisfies all the
properties of the general linear operator in the abstract half-space problem
considered in \cite{Be-21}.

The rest of the paper is organized as follows. In Section $\ref{S2}$, the
model considered is presented. The probabilistic formulation of the
collision operator considered and its relation to a more classical
formulation \cite{BDLP-94, BBBD-18, GP-20} is accounted for in Section $\ref%
{S2.1}$ \ Some classical results for the collision operator in Section $\ref%
{S2.2}$ and the linearized collision operator in Section $\ref{S2.3}$ are
reviewed. Section $\ref{S3}$ is devoted to the main results of this paper. A
proof of compactness of the integral operator is presented in Section $\ref%
{PT1}$, while a proof of the bounds on the collision frequency appears in
Section $\ref{PT2}$,

\section{Kinetic model\label{S2}}

In this section the model considered is presented. A probabilistic
formulation of the collision operator, cf. \cite{Kogan, SNB-85, BPS-90,
Be-21a}, is considered, whose relation to a more classical formulation is
accounted for. Known properties of the model and corresponding linearized
collision operator are also reviewed.

Consider a single species of polyatomic molecules with mass $m$, where the
polyatomicity is modeled by an internal energy variable $I\in $ $\mathbb{R}%
_{+}$. The distribution functions are nonnegative functions of the form $%
f=f\left( t,\mathbf{x},\boldsymbol{\xi },I\right) $, with $t\in \mathbb{R}%
_{+}$, $\mathbf{x}=\left( x,y,z\right) \in \mathbb{R}^{3}$, and $\boldsymbol{%
\xi }=\left( \xi _{x},\xi _{y},\xi _{z}\right) \in \mathbb{R}^{3}$.

Moreover, consider the real Hilbert space $\mathcal{\mathfrak{h:}}%
=L^{2}\left( d\boldsymbol{\xi \,}\mathbf{\,}dI\right) $, with inner product%
\begin{equation*}
\left( f,g\right) =\int_{\mathbb{R}^{3}\times \mathbb{R}_{+}}fg\,d%
\boldsymbol{\xi \,}dI\text{ for }f,g\in L^{2}\left( d\boldsymbol{\xi \,}%
\mathbf{\,}dI\right) \text{.}
\end{equation*}

The evolution of the distribution functions is (in the absence of external
forces) described by the Boltzmann equation%
\begin{equation}
\frac{\partial f}{\partial t}+\left( \boldsymbol{\xi }\cdot \nabla _{\mathbf{%
x}}\right) f=Q\left( f,f\right) \text{,}  \label{BE1}
\end{equation}%
where the collision operator $Q=Q\left( f,f\right) $ is a quadratic bilinear
operator that accounts for the change of velocities and internal energies of
particles due to binary collisions (assuming that the gas is rarefied, such
that other collisions are negligible).

\subsection{Collision operator\label{S2.1}}

The collision operator in the Boltzmann equation $\left( \ref{BE1}\right) $
for polyatomic molecules can be written in the form 
\begin{eqnarray*}
Q(f,f) &=&\int_{\left( \mathbb{R}^{3}\times \mathbb{R}_{+}\right) ^{3}}W(%
\boldsymbol{\xi },\boldsymbol{\xi }_{\ast },I,I_{\ast }\left\vert 
\boldsymbol{\xi }^{\prime },\boldsymbol{\xi }_{\ast }^{\prime },I^{\prime
},I_{\ast }^{\prime }\right. ) \\
&&\times \left( \frac{f^{\prime }f_{\ast }^{\prime }}{\left( I^{\prime
}I_{\ast }^{\prime }\right) ^{\delta /2-1}}-\frac{ff_{\ast }}{\left(
II_{\ast }\right) ^{\delta /2-1}}\right) \,d\boldsymbol{\xi }_{\ast }d%
\boldsymbol{\xi }^{\prime }d\boldsymbol{\xi }_{\ast }^{\prime }dI_{\ast
}dI^{\prime }dI_{\ast }^{\prime }\text{.}
\end{eqnarray*}%
Here and below the abbreviations%
\begin{equation}
f_{\ast }=f\left( t,\mathbf{x},\boldsymbol{\xi }_{\ast },I_{\ast }\right) 
\text{, }f^{\prime }=f\left( t,\mathbf{x},\boldsymbol{\xi }^{\prime
},I^{\prime }\right) \text{, and }f_{\ast }^{\prime }=f\left( t,\mathbf{x},%
\boldsymbol{\xi }_{\ast }^{\prime },I_{\ast }^{\prime }\right)  \label{a1}
\end{equation}%
are used, and $\delta $, with $\delta \geq 2$, denote the number of internal
degrees of freedom.

The transition probabilities $W(\boldsymbol{\xi },\boldsymbol{\xi }_{\ast
},I,I_{\ast }\left\vert \boldsymbol{\xi }^{\prime },\boldsymbol{\xi }_{\ast
}^{\prime },I^{\prime },I_{\ast }^{\prime }\right. )$ are of the form, cf. 
\cite{Kogan, SNB-85, BPS-90} for the monatomic case,%
\begin{eqnarray}
&&W(\boldsymbol{\xi },\boldsymbol{\xi }_{\ast },I,I_{\ast }\left\vert 
\boldsymbol{\xi }^{\prime },\boldsymbol{\xi }_{\ast }^{\prime },I^{\prime
},I_{\ast }^{\prime }\right. )  \notag \\
&=&4m\left( I^{\prime }I_{\ast }^{\prime }\right) ^{\delta /2-1}\sigma
^{\prime }\frac{\left\vert \mathbf{g}^{\prime }\right\vert }{\left\vert 
\mathbf{g}\right\vert }\delta _{3}\left( \boldsymbol{\xi }+\boldsymbol{\xi }%
_{\ast }-\boldsymbol{\xi }^{\prime }-\boldsymbol{\xi }_{\ast }^{\prime
}\right)  \notag \\
&&\times \delta _{1}\left( \frac{m}{2}\left( \left\vert \boldsymbol{\xi }%
\right\vert ^{2}+\left\vert \boldsymbol{\xi }_{\ast }\right\vert
^{2}-\left\vert \boldsymbol{\xi }^{\prime }\right\vert ^{2}-\left\vert 
\boldsymbol{\xi }_{\ast }^{\prime }\right\vert ^{2}\right) -\Delta I\right) 
\notag \\
&=&4m\left( II_{\ast }\right) ^{\delta /2-1}\sigma \frac{\left\vert \mathbf{g%
}\right\vert }{\left\vert \mathbf{g}^{\prime }\right\vert }\delta _{3}\left( 
\boldsymbol{\xi }+\boldsymbol{\xi }_{\ast }-\boldsymbol{\xi }^{\prime }-%
\boldsymbol{\xi }_{\ast }^{\prime }\right)  \notag \\
&&\times \delta _{1}\left( \frac{m}{2}\left( \left\vert \boldsymbol{\xi }%
\right\vert ^{2}+\left\vert \boldsymbol{\xi }_{\ast }\right\vert
^{2}-\left\vert \boldsymbol{\xi }^{\prime }\right\vert ^{2}-\left\vert 
\boldsymbol{\xi }_{\ast }^{\prime }\right\vert ^{2}\right) -\Delta I\right) 
\text{, }  \notag \\
&&\text{where }\sigma ^{\prime }=\sigma \left( \left\vert \mathbf{g}^{\prime
}\right\vert ,\left\vert \cos \theta \right\vert ,I^{\prime },I_{\ast
}^{\prime },I,I_{\ast }\right) >0\text{ and}  \notag \\
&&\sigma =\sigma \left( \left\vert \mathbf{g}\right\vert ,\left\vert \cos
\theta \right\vert ,I,I_{\ast },I^{\prime },I_{\ast }^{\prime }\right) >0%
\text{ \ a.e., with }\cos \theta =\frac{\mathbf{g}\cdot \mathbf{g}^{\prime }%
}{\left\vert \mathbf{g}\right\vert \left\vert \mathbf{g}^{\prime
}\right\vert }\text{,}  \notag \\
&&\text{ }\mathbf{g}=\boldsymbol{\xi }-\boldsymbol{\xi }_{\ast }\text{, }%
\mathbf{g}^{\prime }=\boldsymbol{\xi }^{\prime }-\boldsymbol{\xi }_{\ast
}^{\prime }\text{, and }\Delta I=I^{\prime }+I_{\ast }^{\prime }-I-I_{\ast }%
\text{,}  \label{tp}
\end{eqnarray}%
where $\delta _{3}$ and $\delta _{1}$ denote the Dirac's delta function in $%
\mathbb{R}^{3}$ and $\mathbb{R}$, respectively; taking the conservation of
momentum and total energy into account. Here it is assumed that the
scattering cross sections $\sigma $ satisfy the microreversibility condition%
\begin{eqnarray}
&&\left( II_{\ast }\right) ^{\delta /2-1}\left\vert \mathbf{g}\right\vert
^{2}\sigma \left( \left\vert \mathbf{g}\right\vert ,\left\vert \cos \theta
\right\vert ,I,I_{\ast },I^{\prime },I_{\ast }^{\prime }\right)  \notag \\
&=&\left( I^{\prime }I_{\ast }^{\prime }\right) ^{\delta /2-1}\left\vert 
\mathbf{g}^{\prime }\right\vert ^{2}\sigma \left( \left\vert \mathbf{g}%
^{\prime }\right\vert ,\left\vert \cos \theta \right\vert ,I^{\prime
},I_{\ast }^{\prime },I,I_{\ast }\right) \text{.}  \label{mr}
\end{eqnarray}%
Furthermore, to have invariance of change of particles in a collision, it is
assumed that the scattering cross sections $\sigma $ satisfy the symmetry
relations%
\begin{eqnarray}
\sigma \left( \left\vert \mathbf{g}\right\vert ,\left\vert \cos \theta
\right\vert ,I,I_{\ast },I^{\prime },I_{\ast }^{\prime }\right) &=&\sigma
\left( \left\vert \mathbf{g}\right\vert ,\left\vert \cos \theta \right\vert
,I,I_{\ast },I_{\ast }^{\prime },I^{\prime }\right)  \notag \\
&=&\sigma \left( \left\vert \mathbf{g}\right\vert ,\left\vert \cos \theta
\right\vert ,I_{\ast },I,I_{\ast }^{\prime },I^{\prime }\right) \text{.}
\label{sr}
\end{eqnarray}%
The invariance under change of particles in a collision, which follows by
the definition of the transition probability $\left( \ref{tp}\right) $ and
the symmetry relations $\left( \ref{sr}\right) $ for the collision
frequency, and the microreversibility of the collisions $\left( \ref{mr}%
\right) $, imply that the transition probabilities $\left( \ref{tp}\right) $
satisfy the relations%
\begin{eqnarray}
W(\boldsymbol{\xi },\boldsymbol{\xi }_{\ast },I,I_{\ast }\left\vert 
\boldsymbol{\xi }^{\prime },\boldsymbol{\xi }_{\ast }^{\prime },I^{\prime
},I_{\ast }^{\prime }\right. ) &=&W(\boldsymbol{\xi }_{\ast },\boldsymbol{%
\xi },I_{\ast },I\left\vert \boldsymbol{\xi }_{\ast }^{\prime },\boldsymbol{%
\xi }^{\prime },I_{\ast }^{\prime },I^{\prime }\right. )  \notag \\
W(\boldsymbol{\xi },\boldsymbol{\xi }_{\ast },I,I_{\ast }\left\vert 
\boldsymbol{\xi }^{\prime },\boldsymbol{\xi }_{\ast }^{\prime },I^{\prime
},I_{\ast }^{\prime }\right. ) &=&W(\boldsymbol{\xi }^{\prime },\boldsymbol{%
\xi }_{\ast }^{\prime },I^{\prime },I_{\ast }^{\prime }\left\vert 
\boldsymbol{\xi },\boldsymbol{\xi }_{\ast },I,I_{\ast }\right. )  \notag \\
W(\boldsymbol{\xi },\boldsymbol{\xi }_{\ast },I,I_{\ast }\left\vert 
\boldsymbol{\xi }^{\prime },\boldsymbol{\xi }_{\ast }^{\prime },I^{\prime
},I_{\ast }^{\prime }\right. ) &=&W(\boldsymbol{\xi },\boldsymbol{\xi }%
_{\ast },I,I_{\ast }\left\vert \boldsymbol{\xi }_{\ast }^{\prime },%
\boldsymbol{\xi }^{\prime },I_{\ast }^{\prime },I^{\prime }\right. )\text{.}
\label{rel1}
\end{eqnarray}%
Applying known properties of Dirac's delta function, the transition
probabilities may be transformed to 
\begin{align*}
& W(\boldsymbol{\xi },\boldsymbol{\xi }_{\ast },I,I_{\ast }\left\vert 
\boldsymbol{\xi }^{\prime },\boldsymbol{\xi }_{\ast }^{\prime },I^{\prime
},I_{\ast }^{\prime }\right. ) \\
=& \frac{m}{2}\left( I^{\prime }I_{\ast }^{\prime }\right) ^{\delta
/2-1}\sigma ^{\prime }\frac{\left\vert \mathbf{g}^{\prime }\right\vert }{%
\left\vert \mathbf{g}\right\vert }\delta _{3}\left( \mathbf{G}-\mathbf{G}%
^{\prime }\right) \delta _{1}\left( \frac{m}{4}\left( \left\vert \mathbf{g}%
\right\vert ^{2}-\left\vert \mathbf{g}^{\prime }\right\vert ^{2}\right)
-\Delta I\right) \\
=\,& \frac{m}{2}\left( I^{\prime }I_{\ast }^{\prime }\right) ^{\delta
/2-1}\sigma ^{\prime }\frac{\left\vert \mathbf{g}^{\prime }\right\vert }{%
\left\vert \mathbf{g}\right\vert }\delta _{3}\left( \mathbf{G}-\mathbf{G}%
^{\prime }\right) \delta _{1}\left( E-E^{\prime }\right) \\
=& \,\frac{\left( I^{\prime }I_{\ast }^{\prime }\right) ^{\delta /2-1}}{%
\left\vert \mathbf{g}\right\vert }\sigma ^{\prime }\delta _{3}\left( \mathbf{%
G}-\mathbf{G}^{\prime }\right) \delta _{1}\left( \sqrt{\left\vert \mathbf{g}%
\right\vert ^{2}-\frac{4}{m}\Delta I}-\left\vert \mathbf{g}^{\prime
}\right\vert \right) \mathbf{1}_{m\left\vert \mathbf{g}\right\vert
^{2}>4\Delta I} \\
=& \left( II_{\ast }\right) ^{\delta /2-1}\sigma \frac{\left\vert \mathbf{g}%
\right\vert }{\left\vert \mathbf{g}^{\prime }\right\vert ^{2}}\delta
_{3}\left( \mathbf{G}-\mathbf{G}^{\prime }\right) \delta _{1}\left( \sqrt{%
\left\vert \mathbf{g}\right\vert ^{2}-\frac{4}{m}\Delta I}-\left\vert 
\mathbf{g}^{\prime }\right\vert \right) \mathbf{1}_{m\left\vert \mathbf{g}%
\right\vert ^{2}>4\Delta I}\text{, with} \\
& \mathbf{G}=\frac{\boldsymbol{\xi }+\boldsymbol{\xi }_{\ast }}{2},\text{ }%
\mathbf{G}^{\prime }=\frac{\boldsymbol{\xi }^{\prime }+\boldsymbol{\xi }%
_{\ast }^{\prime }}{2}\text{, }E=\frac{m}{4}\left\vert \mathbf{g}\right\vert
^{2}+I+I_{\ast }\text{,}\ E^{\prime }=\frac{m}{4}\left\vert \mathbf{g}%
^{\prime }\right\vert ^{2}+I^{\prime }+I_{\ast }^{\prime }\text{,}
\end{align*}

A series of change of variables: first $\left\{ \boldsymbol{\xi }^{\prime },%
\boldsymbol{\xi }_{\ast }^{\prime }\right\} \rightarrow \!\left\{ \mathbf{g}%
^{\prime }=\boldsymbol{\xi }^{\prime }-\boldsymbol{\xi }_{\ast }^{\prime }%
\text{,}\mathbf{G}^{\prime }=\dfrac{\boldsymbol{\xi }^{\prime }+\boldsymbol{%
\xi }_{\ast }^{\prime }}{2}\!\right\} \!$, followed by a change to spherical
coordinates $\left\{ \mathbf{g}^{\prime }\right\} \rightarrow \left\{
\left\vert \mathbf{g}^{\prime }\right\vert ,\boldsymbol{\omega \,}=\dfrac{%
\mathbf{g}^{\prime }}{\left\vert \mathbf{g}^{\prime }\right\vert }\right\} $%
, and then $\left\{ \left\vert \mathbf{g}^{\prime }\right\vert ,I^{\prime
},I_{\ast }^{\prime }\right\} \rightarrow \left\{ R=\dfrac{m\left\vert 
\mathbf{g}^{\prime }\right\vert ^{2}}{4E},r=\dfrac{I^{\prime }}{(1-R)E}%
,E^{\prime }=\dfrac{m}{4}\left\vert \mathbf{g}^{\prime }\right\vert
^{2}+I^{\prime }+I_{\ast }^{\prime }\right\} $, observing that%
\begin{eqnarray}
&&d\boldsymbol{\xi }^{\prime }d\boldsymbol{\xi }_{\ast }^{\prime }dI^{\prime
}dI_{\ast }^{\prime }=d\mathbf{g}^{\prime }d\mathbf{G}^{\prime }dI^{\prime
}dI_{\ast }^{\prime }=\left\vert \mathbf{g}^{\prime }\right\vert
^{2}d\left\vert \mathbf{g}^{\prime }\right\vert d\boldsymbol{\omega \,}d%
\mathbf{G}^{\prime }dI^{\prime }dI_{\ast }^{\prime }  \notag \\
&=&\frac{4}{m^{3/2}}E^{5/2}(1-R)R^{1/2}dRd\boldsymbol{\omega \,}d\mathbf{G}%
^{\prime }\boldsymbol{\,}dr\boldsymbol{\,}dE^{\prime }  \notag \\
&=&\frac{4E^{\delta +1/2}}{m^{3/2}\left( I^{\prime }I_{\ast }^{\prime
}\right) ^{\delta /2-1}}\left( r(1-r)\right) ^{\delta /2-1}\left( 1-R\right)
^{\delta -1}R^{1/2}\boldsymbol{\,}dE^{\prime }d\mathbf{G}^{\prime }d%
\boldsymbol{\omega \,}dr\boldsymbol{\,}dR\text{,}  \notag \\
&&\text{ where }I^{\prime }=r\left( 1-R\right) E\text{ and }I_{\ast
}^{\prime }=\left( 1-r\right) \left( 1-R\right) E\text{,}  \label{df1}
\end{eqnarray}%
results in a more familiar form of the Boltzmann collision operator for
polyatomic molecules modeled with a continuous internal energy variable \cite%
{BDLP-94, BBBD-18, GP-20} 
\begin{eqnarray*}
Q(f,f) &=&\int_{\left( \mathbb{R}^{3}\right) ^{2}\times \left( \mathbb{R}%
_{+}\right) ^{2}\times \lbrack 0,1]^{2}\mathbb{\times S}^{2}}W(\boldsymbol{%
\xi },\boldsymbol{\xi }_{\ast },I,I_{\ast }\left\vert \boldsymbol{\xi }%
^{\prime },\boldsymbol{\xi }_{\ast }^{\prime },I^{\prime },I_{\ast }^{\prime
}\right. ) \\
&&\times \frac{4E^{\delta +1/2}}{m^{3/2}\left( I^{\prime }I_{\ast }^{\prime
}\right) ^{\delta /2-1}}\left( \frac{f^{\prime }f_{\ast }^{\prime }}{\left(
I^{\prime }I_{\ast }^{\prime }\right) ^{\delta /2-1}}-\frac{ff_{\ast }}{%
\left( II_{\ast }\right) ^{\delta /2-1}}\right) \\
&&\times \left( r(1-r)\right) ^{\delta /2-1}\left( 1-R\right) ^{\delta
-1}R^{1/2}\boldsymbol{\,}dE^{\prime }d\mathbf{G}^{\prime }d\boldsymbol{%
\omega \,}dr\boldsymbol{\,}dR\boldsymbol{\,}d\boldsymbol{\xi }_{\ast }%
\boldsymbol{\,}dI_{\ast } \\
&=&\int_{\mathbb{R}^{3}\times \mathbb{R}_{+}\times \lbrack 0,1]^{2}\mathbb{%
\times S}^{2}}B\left( \frac{f^{\prime }f_{\ast }^{\prime }}{\left( I^{\prime
}I_{\ast }^{\prime }\right) ^{\delta /2-1}}-\frac{ff_{\ast }}{\left(
II_{\ast }\right) ^{\delta /2-1}}\right) \\
&&\times \left( r(1-r)\right) ^{\delta /2-1}\left( 1-R\right) ^{\delta
-1}R^{1/2}\left( II_{\ast }\right) ^{\delta /2-1}d\boldsymbol{\omega \,}dr%
\boldsymbol{\,}dR\boldsymbol{\,}d\boldsymbol{\xi }_{\ast }\boldsymbol{\,}%
dI_{\ast },
\end{eqnarray*}%
with the collision kernel 
\begin{eqnarray}
B &=&\frac{2\sigma \left\vert \mathbf{g}\right\vert E^{\delta +1/2}\mathbf{1}%
_{m\left\vert \mathbf{g}\right\vert ^{2}>4\Delta I}}{\sqrt{m}\sqrt{%
\left\vert \mathbf{g}\right\vert ^{2}-\frac{4}{m}\Delta I}\left( I^{\prime
}I_{\ast }^{\prime }\right) ^{\delta /2-1}}=\frac{\sigma \left\vert \mathbf{g%
}\right\vert E^{2}\mathbf{1}_{E>0}}{\left( 1-R\right) ^{\delta
-2}R^{1/2}\left( r(1-r)\right) ^{\delta /2-1}}\text{, where}  \notag \\
E &=&\frac{m\left\vert \mathbf{g}\right\vert ^{2}}{4}-\Delta I\text{ and }%
\Delta I=I^{\prime }+I_{\ast }^{\prime }-I-I_{\ast }\text{,}  \label{ck1}
\end{eqnarray}%
and%
\begin{equation*}
\left\{ 
\begin{array}{l}
\boldsymbol{\xi }^{\prime }=\dfrac{\boldsymbol{\xi }+\boldsymbol{\xi }_{\ast
}}{2}+\dfrac{\sqrt{\left\vert \boldsymbol{\xi }-\boldsymbol{\xi }_{\ast
}\right\vert ^{2}-\dfrac{4}{m}\Delta I}}{2}\omega =\mathbf{G}+\dfrac{\sqrt{%
\left\vert \mathbf{g}\right\vert ^{2}-\dfrac{4}{m}\Delta I}}{2}\omega
\medskip \\ 
\boldsymbol{\xi }_{\ast }^{\prime }=\dfrac{\boldsymbol{\xi }+\boldsymbol{\xi 
}_{\ast }}{2}-\dfrac{\sqrt{\left\vert \boldsymbol{\xi }-\boldsymbol{\xi }%
_{\ast }\right\vert ^{2}-\dfrac{4}{m}\Delta I}}{2}\omega =\mathbf{G}-\dfrac{%
\sqrt{\left\vert \mathbf{g}\right\vert ^{2}-\dfrac{4}{m}\Delta I}}{2}\omega%
\end{array}%
\right. \text{, }\omega \in S^{2}\text{.}
\end{equation*}

\begin{remark}
\label{RRC}By multiplying the transition probability $\left( \ref{tp}\right) 
$ with an extra Dirac's delta function in $\mathbb{R}$, namely%
\begin{equation*}
\delta _{1}\left( \Delta I\right)
\end{equation*}%
we obtain the case where the molecules are assumed to undergo resonant
collisions \cite{BRS-22}.
\end{remark}

\subsection{Collision invariants and Maxwellian distributions\label{S2.2}}

The following lemma follows directly by the relations $\left( \ref{rel1}%
\right) $.

\begin{lemma}
\label{L0}The measure 
\begin{equation*}
dA=W(\boldsymbol{\xi },\boldsymbol{\xi }_{\ast },I,I_{\ast }\left\vert 
\boldsymbol{\xi }^{\prime },\boldsymbol{\xi }_{\ast }^{\prime },I^{\prime
},I_{\ast }^{\prime }\right. )\,d\boldsymbol{\xi }\,d\boldsymbol{\xi }_{\ast
}d\boldsymbol{\xi }^{\prime }d\boldsymbol{\xi }_{\ast }^{\prime }dIdI_{\ast
}dI^{\prime }dI_{\ast }^{\prime }
\end{equation*}%
is invariant under the interchanges of variables%
\begin{eqnarray}
\left( \boldsymbol{\xi },\boldsymbol{\xi }_{\ast },I,I_{\ast }\right)
&\leftrightarrow &\left( \boldsymbol{\xi }^{\prime },\boldsymbol{\xi }_{\ast
}^{\prime },I^{\prime },I_{\ast }^{\prime }\right) \text{,}  \notag \\
\left( \boldsymbol{\xi },I\right) &\leftrightarrow &\left( \boldsymbol{\xi }%
_{\ast },I_{\ast }\right) \text{, and}  \notag \\
\left( \boldsymbol{\xi }^{\prime },I^{\prime }\right) &\leftrightarrow
&\left( \boldsymbol{\xi }_{\ast }^{\prime },I_{\ast }^{\prime }\right) \text{%
,}  \label{tr1}
\end{eqnarray}%
respectively.
\end{lemma}

The weak form of the collision operator $Q(f,f)$ reads%
\begin{eqnarray*}
\left( Q(f,f),g\right) &=&\int_{\left( \mathbb{R}^{3}\times \mathbb{R}%
_{+}\right) ^{4}}\left( \frac{f^{\prime }f_{\ast }^{\prime }}{\left(
I^{\prime }I_{\ast }^{\prime }\right) ^{\delta /2-1}}-\frac{ff_{\ast }}{%
\left( II_{\ast }\right) ^{\delta /2-1}}\right) g\,dA \\
&=&\int_{\left( \mathbb{R}^{3}\times \mathbb{R}_{+}\right) ^{4}}\left( \frac{%
f^{\prime }f_{\ast }^{\prime }}{\left( I^{\prime }I_{\ast }^{\prime }\right)
^{\delta /2-1}}-\frac{ff_{\ast }}{\left( II_{\ast }\right) ^{\delta /2-1}}%
\right) g_{\ast }\,dA \\
&=&-\int_{\left( \mathbb{R}^{3}\times \mathbb{R}_{+}\right) ^{4}}\left( 
\frac{f^{\prime }f_{\ast }^{\prime }}{\left( I^{\prime }I_{\ast }^{\prime
}\right) ^{\delta /2-1}}-\frac{ff_{\ast }}{\left( II_{\ast }\right) ^{\delta
/2-1}}\right) g^{\prime }\,dA \\
&=&-\int_{\left( \mathbb{R}^{3}\times \mathbb{R}_{+}\right) ^{4}}\left( 
\frac{f^{\prime }f_{\ast }^{\prime }}{\left( I^{\prime }I_{\ast }^{\prime
}\right) ^{\delta /2-1}}-\frac{ff_{\ast }}{\left( II_{\ast }\right) ^{\delta
/2-1}}\right) g_{\ast }^{\prime }\,dA,
\end{eqnarray*}%
for $g=g\left( \boldsymbol{\xi },I\right) $, such that the first integral is
defined, while the following integrals are obtained by applying Lemma $\ref%
{L0}$.

We have the following proposition.

\begin{proposition}
\label{P1}Let $g=g\left( \boldsymbol{\xi },I\right) $ be such that%
\begin{equation*}
\int_{\left( \mathbb{R}^{3}\times \mathbb{R}_{+}\right) ^{4}}\left( \frac{%
f^{\prime }f_{\ast }^{\prime }}{\left( I^{\prime }I_{\ast }^{\prime }\right)
^{\delta /2-1}}-\frac{ff_{\ast }}{\left( II_{\ast }\right) ^{\delta /2-1}}%
\right) g\,dA
\end{equation*}
is defined. Then%
\begin{equation*}
\left( Q(f,f),g\right) =\frac{1}{4}\int_{\left( \mathbb{R}^{3}\times \mathbb{%
R}_{+}\right) ^{4}}\left( \frac{f^{\prime }f_{\ast }^{\prime }}{\left(
I^{\prime }I_{\ast }^{\prime }\right) ^{\delta /2-1}}-\frac{ff_{\ast }}{%
\left( II_{\ast }\right) ^{\delta /2-1}}\right) \left( g+g_{\ast }-g^{\prime
}-g_{\ast }^{\prime }\right) \,dA.
\end{equation*}
\end{proposition}

\begin{definition}
A function $g=g\left( \boldsymbol{\xi },I\right) \ $is a collision invariant
if%
\begin{equation*}
\left( g+g_{\ast }-g^{\prime }-g_{\ast }^{\prime }\right) W(\boldsymbol{\xi }%
,\boldsymbol{\xi }_{\ast },I,I_{\ast }\left\vert \boldsymbol{\xi }^{\prime },%
\boldsymbol{\xi }_{\ast }^{\prime },I^{\prime },I_{\ast }^{\prime }\right.
)=0\text{ a.e. .}
\end{equation*}
\end{definition}

Then it is clear that $1,$ $\xi _{x},$ $\xi _{y},$ $\xi _{z},$ and $%
m\left\vert \boldsymbol{\xi }\right\vert ^{2}+2I$ are collision invariants -
corresponding to conservation of mass, momentum, and total energy - and, in
fact we have the following proposition, cf. Proposition 1 in \cite{BDLP-94}.

\begin{proposition}
\label{P2}The vector space of collision invariants is generated by 
\begin{equation*}
\left\{ 1,\xi _{x},\xi _{y},\xi _{z},m\left\vert \boldsymbol{\xi }%
\right\vert ^{2}+2I\right\} \text{.}
\end{equation*}
\end{proposition}

Define%
\begin{equation*}
\mathcal{W}\left[ f\right] :=\left( Q(f,f),\log \left( I^{1-\delta
/2}f\right) \right) .
\end{equation*}%
It follows by Proposition $\ref{P1}$ that%
\begin{eqnarray*}
&&\mathcal{W}\left[ f\right] \\
&=&-\frac{1}{4}\int\limits_{\left( \mathbb{R}^{3}\times \mathbb{R}%
_{+}\right) ^{4}}\!\frac{ff_{\ast }}{\left( II_{\ast }\right) ^{\delta /2-1}}%
\left( \frac{\left( II_{\ast }\right) ^{\delta /2-1}f^{\prime }f_{\ast
}^{\prime }}{ff_{\ast }\left( I^{\prime }I_{\ast }^{\prime }\right) ^{\delta
/2-1}}-1\right) \log \left( \frac{\left( II_{\ast }\right) ^{\delta
/2-1}f^{\prime }f_{\ast }^{\prime }}{ff_{\ast }\left( I^{\prime }I_{\ast
}^{\prime }\right) ^{\delta /2-1}}\right) dA\text{.}
\end{eqnarray*}%
Since $\left( x-1\right) \mathrm{log}\left( x\right) \geq 0$ for $x>0$, with
equality if and only if $x=1$,%
\begin{equation*}
\mathcal{W}\left[ f\right] \leq 0\text{,}
\end{equation*}%
with equality if and only if 
\begin{equation}
\left( \frac{ff_{\ast }}{\left( II_{\ast }\right) ^{\delta /2-1}}-\frac{%
f^{\prime }f_{\ast }^{\prime }}{\left( I^{\prime }I_{\ast }^{\prime }\right)
^{\delta /2-1}}\right) W(\boldsymbol{\xi },\boldsymbol{\xi }_{\ast
},I,I_{\ast }\left\vert \boldsymbol{\xi }^{\prime },\boldsymbol{\xi }_{\ast
}^{\prime },I^{\prime },I_{\ast }^{\prime }\right. )=0\text{ a.e.,}
\label{m1}
\end{equation}%
or, equivalently, if and only if%
\begin{equation*}
Q(f,f)\equiv 0\text{.}
\end{equation*}

For any equilibrium, or Maxwellian, distribution $M$, $Q(M,M)\equiv 0$, why
it follows, by the relation $\left( \ref{m1}\right) $, that 
\begin{multline*}
\left( \log \frac{M}{I^{\delta /2-1}}+\log \frac{M_{\ast }}{I_{\ast
}^{\delta /2-1}}-\log \frac{M^{\prime }}{\left( I^{\prime }\right) ^{\delta
/2-1}}-\log \frac{M_{\ast }^{\prime }}{\left( I_{\ast }^{\prime }\right)
^{\delta /2-1}}\right) \\
\times W(\boldsymbol{\xi },\boldsymbol{\xi }_{\ast },I,I_{\ast }\left\vert 
\boldsymbol{\xi }^{\prime },\boldsymbol{\xi }_{\ast }^{\prime },I^{\prime
},I_{\ast }^{\prime }\right. )=0\text{ a.e. .}
\end{multline*}%
Hence, $\log \dfrac{M}{I^{\delta /2-1}}$ is a collision invariant, and the
Maxwellian distributions are of the form 
\begin{equation*}
M=\dfrac{nI^{\delta /2-1}m^{3/2}}{\left( 2\pi \right) ^{3/2}T^{\left( \delta
+3\right) /2}\Gamma \left( \frac{\delta }{2}\right) }e^{-\left( m\left\vert 
\boldsymbol{\xi }-\mathbf{u}\right\vert ^{2}+2I\right) /\left( 2T\right) }%
\text{, }
\end{equation*}%
where $n=\left( M,1\right) $, $\mathbf{u}=\dfrac{1}{n}\left( M,\boldsymbol{%
\xi }\right) $, and $T=\dfrac{m}{3n}\left( M,\left\vert \boldsymbol{\xi }-%
\mathbf{u}\right\vert ^{2}\right) $, while $\Gamma =\Gamma (s)$ denotes the
Gamma function $\Gamma (s)=\int_{0}^{\infty }x^{s-1}e^{-x}\,dx$.

Note that by equation $\left( \ref{m1}\right) $ any Maxwellian distribution $%
M$ satisfies the relations 
\begin{equation}
\left( \frac{M^{\prime }M_{\ast }^{\prime }}{\left( I^{\prime }I_{\ast
}^{\prime }\right) ^{\delta /2-1}}-\frac{MM_{\ast }}{\left( II_{\ast
}\right) ^{\delta /2-1}}\right) W(\boldsymbol{\xi },\boldsymbol{\xi }_{\ast
},I,I_{\ast }\left\vert \boldsymbol{\xi }^{\prime },\boldsymbol{\xi }_{\ast
}^{\prime },I^{\prime },I_{\ast }^{\prime }\right. )=0\text{ a.e. }.
\label{M1}
\end{equation}

\begin{remark}
Introducing the $\mathcal{H}$-functional%
\begin{equation*}
\mathcal{H}\left[ f\right] =\left( f,\log \left( I^{1-\delta /2}f\right)
\right) \text{,}
\end{equation*}%
an $\mathcal{H}$-theorem can be obtained, cf. \cite{BDLP-94, BBBD-18, GP-20}.
\end{remark}

\subsection{Linearized collision operator\label{S2.3}}

Considering deviations of a Maxwellian distribution $M=\dfrac{I^{\delta
/2-1}m^{3/2}}{\left( 2\pi \right) ^{3/2}\Gamma \left( \frac{\delta }{2}%
\right) }e^{-\frac{m\left\vert \boldsymbol{\xi }\right\vert ^{2}}{2}-I}$ of
the form%
\begin{equation}
f=M+M^{1/2}h  \label{s1}
\end{equation}%
results, by insertion in the Boltzmann equation $\left( \ref{BE1}\right) $,
in the equation%
\begin{equation}
\frac{\partial h}{\partial t}+\left( \boldsymbol{\xi }\cdot \nabla _{\mathbf{%
x}}\right) h+\mathcal{L}h=\Gamma \left( h,h\right) \text{,}  \label{LBE}
\end{equation}%
where the linearized collision operator $\mathcal{L}$ is given by \ 
\begin{eqnarray}
\mathcal{L}h &=&-\!M^{-1/2}\left( Q(M,M^{1/2}h)+Q(M^{1/2}h,M)\right)  \notag
\\
&=&\!M^{-1/2}\int_{\left( \mathbb{R}^{3}\times \mathbb{R}_{+}\right) ^{3}}%
\frac{\left( MM_{\ast }M^{\prime }M_{\ast }^{\prime }\right) ^{1/2}}{\left(
II_{\ast }I^{\prime }I_{\ast }^{\prime }\right) ^{\delta /4-1/2}}W(%
\boldsymbol{\xi },\boldsymbol{\xi }_{\ast },I,I_{\ast }\left\vert 
\boldsymbol{\xi }^{\prime },\boldsymbol{\xi }_{\ast }^{\prime },I^{\prime
},I_{\ast }^{\prime }\right. )  \notag \\
&&\!\times \left( \frac{h}{M^{1/2}}+\frac{h_{\ast }}{M_{\ast }^{1/2}}-\frac{%
h^{\prime }}{\left( M^{\prime }\right) ^{1/2}}-\frac{h_{\ast }^{\prime }}{%
\left( M_{\ast }^{\prime }\right) ^{1/2}}\right) \,d\boldsymbol{\xi }_{\ast
}d\boldsymbol{\xi }^{\prime }d\boldsymbol{\xi }_{\ast }^{\prime }dI_{\ast
}dI^{\prime }dI_{\ast }^{\prime }  \notag \\
&=&\upsilon h-K\left( h\right) \text{,}  \label{dec2}
\end{eqnarray}%
with%
\begin{eqnarray}
\upsilon &=&\int_{\left( \mathbb{R}^{3}\times \mathbb{R}_{+}\right) ^{3}}\!%
\frac{M_{\ast }}{\left( II_{\ast }\right) ^{\delta /2-1}}W(\boldsymbol{\xi },%
\boldsymbol{\xi }_{\ast },I,I_{\ast }\left\vert \boldsymbol{\xi }^{\prime },%
\boldsymbol{\xi }_{\ast }^{\prime },I^{\prime },I_{\ast }^{\prime }\right. )d%
\boldsymbol{\xi }_{\ast }d\boldsymbol{\xi }^{\prime }d\boldsymbol{\xi }%
_{\ast }^{\prime }dI_{\ast }dI^{\prime }dI_{\ast }^{\prime }\!\text{,} 
\notag \\
K\left( h\right) &=&M^{-1/2}\int_{\left( \mathbb{R}^{3}\times \mathbb{R}%
_{+}\right) ^{3}}\frac{\left( MM_{\ast }M^{\prime }M_{\ast }^{\prime
}\right) ^{1/2}}{\left( II_{\ast }I^{\prime }I_{\ast }^{\prime }\right)
^{\delta /4-1/2}}W(\boldsymbol{\xi },\boldsymbol{\xi }_{\ast },I,I_{\ast
}\left\vert \boldsymbol{\xi }^{\prime },\boldsymbol{\xi }_{\ast }^{\prime
},I^{\prime },I_{\ast }^{\prime }\right. )  \notag \\
&&\times \left( \frac{h^{\prime }}{\left( M^{\prime }\right) ^{1/2}}+\frac{%
h_{\ast }^{\prime }}{\left( M_{\ast }^{\prime }\right) ^{1/2}}-\frac{h_{\ast
}}{M_{\ast }^{1/2}}\right) \,d\boldsymbol{\xi }_{\ast }d\boldsymbol{\xi }%
^{\prime }d\boldsymbol{\xi }_{\ast }^{\prime }dI_{\ast }dI^{\prime }dI_{\ast
}^{\prime }\text{,}  \label{dec1}
\end{eqnarray}%
while the quadratic term $\Gamma $ is given by%
\begin{equation}
\Gamma \left( h,h\right) =M^{-1/2}Q(M^{1/2}h,M^{1/2}h)\text{.}  \label{nl1}
\end{equation}%
The following lemma follows immediately by Lemma $\ref{L0}$.

\begin{lemma}
\label{L1}The measure 
\begin{equation*}
d\widetilde{A}=\frac{\left( MM_{\ast }M^{\prime }M_{\ast }^{\prime }\right)
^{1/2}}{\left( II_{\ast }I^{\prime }I_{\ast }^{\prime }\right) ^{\delta
/4-1/2}}dA
\end{equation*}%
is invariant under the interchanges of variables $\left( \ref{tr1}\right) $,
respectively.
\end{lemma}

The weak form of the linearized collision operator $\mathcal{L}$ reads%
\begin{eqnarray*}
\left( \mathcal{L}h,g\right) &=&\int_{\left( \mathbb{R}^{3}\times \mathbb{R}%
_{+}\right) ^{4}}\left( \frac{h}{M^{1/2}}+\frac{h_{\ast }}{M_{\ast }^{1/2}}-%
\frac{h^{\prime }}{\left( M^{\prime }\right) ^{1/2}}-\frac{h_{\ast }^{\prime
}}{\left( M_{\ast }^{\prime }\right) ^{1/2}}\right) \frac{g}{M^{1/2}}\,d%
\widetilde{A} \\
&=&\int_{\left( \mathbb{R}^{3}\times \mathbb{R}_{+}\right) ^{4}}\left( \frac{%
h}{M^{1/2}}+\frac{h_{\ast }}{M_{\ast }^{1/2}}-\frac{h^{\prime }}{\left(
M^{\prime }\right) ^{1/2}}-\frac{h_{\ast }^{\prime }}{\left( M_{\ast
}^{\prime }\right) ^{1/2}}\right) \frac{g_{\ast }}{M_{\ast }^{1/2}}\,d%
\widetilde{A} \\
&=&-\int_{\left( \mathbb{R}^{3}\times \mathbb{R}_{+}\right) ^{4}}\left( 
\frac{h}{M^{1/2}}+\frac{h_{\ast }}{M_{\ast }^{1/2}}-\frac{h^{\prime }}{%
\left( M^{\prime }\right) ^{1/2}}-\frac{h_{\ast }^{\prime }}{\left( M_{\ast
}^{\prime }\right) ^{1/2}}\right) \frac{g^{\prime }}{\left( M^{\prime
}\right) ^{1/2}}\,d\widetilde{A} \\
&=&-\int_{\left( \mathbb{R}^{3}\times \mathbb{R}_{+}\right) ^{4}}\left( 
\frac{h}{M^{1/2}}+\frac{h_{\ast }}{M_{\ast }^{1/2}}-\frac{h^{\prime }}{%
\left( M^{\prime }\right) ^{1/2}}-\frac{h_{\ast }^{\prime }}{\left( M_{\ast
}^{\prime }\right) ^{1/2}}\right) \frac{g_{\ast }^{\prime }}{\left( M_{\ast
}^{\prime }\right) ^{1/2}}\,d\widetilde{A}
\end{eqnarray*}%
for $g=g\left( \boldsymbol{\xi },I\right) $, such that the first integral is
defined, while the following integrals are obtained by applying Lemma $\ref%
{L1}$. Then we have the following lemma.

\begin{lemma}
\label{L2}Let $g=g\left( \boldsymbol{\xi },I\right) $ be such that%
\begin{equation*}
\int_{\left( \mathbb{R}^{3}\times \mathbb{R}_{+}\right) ^{4}}\left( \frac{h}{%
M^{1/2}}+\frac{h_{\ast }}{M_{\ast }^{1/2}}-\frac{h^{\prime }}{\left(
M^{\prime }\right) ^{1/2}}-\frac{h_{\ast }^{\prime }}{\left( M_{\ast
}^{\prime }\right) ^{1/2}}\right) \frac{g}{M^{1/2}}\,d\widetilde{A}
\end{equation*}%
is defined. Then%
\begin{eqnarray*}
\left( \mathcal{L}h,g\right) &=&\frac{1}{4}\int_{\left( \mathbb{R}^{3}\times 
\mathbb{R}_{+}\right) ^{4}}\left( \frac{h}{M^{1/2}}+\frac{h_{\ast }}{M_{\ast
}^{1/2}}-\frac{h^{\prime }}{\left( M^{\prime }\right) ^{1/2}}-\frac{h_{\ast
}^{\prime }}{\left( M_{\ast }^{\prime }\right) ^{1/2}}\right) \\
&&\times \left( \frac{g}{M^{1/2}}+\frac{g_{\ast }}{M_{\ast }^{1/2}}-\frac{%
g^{\prime }}{\left( M^{\prime }\right) ^{1/2}}-\frac{g_{\ast }^{\prime }}{%
\left( M_{\ast }^{\prime }\right) ^{1/2}}\right) d\widetilde{A}.
\end{eqnarray*}
\end{lemma}

\begin{proposition}
\label{Prop1}The linearized collision operator is symmetric and nonnegative,%
\begin{equation*}
\left( \mathcal{L}h,g\right) =\left( h,\mathcal{L}g\right) \text{ and }%
\left( \mathcal{L}h,h\right) \geq 0\text{,}
\end{equation*}%
and\ the kernel of $\mathcal{L}$, $\ker \mathcal{L}$, is generated by%
\begin{equation*}
\left\{ M^{1/2},\xi _{x}M^{1/2},\xi _{y}M^{1/2},\xi _{z}M^{1/2},\left(
m\left\vert \boldsymbol{\xi }\right\vert ^{2}+2I\right) M^{1/2}\right\} 
\text{.}
\end{equation*}
\end{proposition}

\begin{proof}
By Lemma $\ref{L2}$, it is immediate that $\left( \mathcal{L}h,g\right)
=\left( h,\mathcal{L}g\right) $ and $\left( \mathcal{L}h,h\right) \geq 0.$
Furthermore, $h\in \ker \mathcal{L}$ if and only if $\left( \mathcal{L}%
h,h\right) =0$, which will be fulfilled\ if and only if%
\begin{equation*}
\left( \frac{h}{M^{1/2}}+\frac{h_{\ast }}{M_{\ast }^{1/2}}-\frac{h^{\prime }%
}{\left( M^{\prime }\right) ^{1/2}}-\frac{h_{\ast }^{\prime }}{\left(
M_{\ast }^{\prime }\right) ^{1/2}}\right) W(\boldsymbol{\xi },\boldsymbol{%
\xi }_{\ast },I,I_{\ast }\left\vert \boldsymbol{\xi }^{\prime },\boldsymbol{%
\xi }_{\ast }^{\prime },I^{\prime },I_{\ast }^{\prime }\right. )=0\text{
a.e.,}
\end{equation*}%
i.e. if and only if $\dfrac{h}{M^{1/2}}$ is a collision invariant. The last
part of the lemma now follows by Proposition $\ref{P2}.$
\end{proof}

\begin{remark}
Note also that the quadratic term is orthogonal to the kernel of $\mathcal{L}
$, i.e. $\Gamma \left( h,h\right) \in \left( \ker \mathcal{L}\right) ^{\perp
_{\mathcal{\mathfrak{h}}}}$.
\end{remark}

\section{Main results\label{S3}}

In this section the main results, concerning compactness properties in
Theorem \ref{Thm1} and bounds on collision frequencies in Theorem \ref{Thm2}%
, are presented. The proofs of the corollaries are essentially the same as
the corresponding ones in \cite{Be-21a}, but are presented here for
self-containment of the paper.

Assume that for some positive number $\gamma $, such that $0<\gamma <1$,
there is a bound%
\begin{eqnarray}
&&\!\!\!\!\!\!\!\!0\leq \sigma \left( \left\vert \mathbf{g}\right\vert ,\cos
\theta ,I,I_{\ast },I^{\prime },I_{\ast }^{\prime }\right) \mathbf{1}%
_{m\left\vert \mathbf{g}\right\vert ^{2}>4\Delta I}\leq C\frac{\Psi +\Psi
^{\gamma /2}}{\left\vert \mathbf{g}\right\vert ^{2}}\frac{\left( I^{\prime
}I_{\ast }^{\prime }\right) ^{\delta /2-1}}{E^{\delta -1/2}}\mathbf{1}%
_{m\left\vert \mathbf{g}\right\vert ^{2}>4\Delta I}\text{, }  \notag \\
&&\text{ with }E=\frac{m}{4}\left\vert \mathbf{g}\right\vert ^{2}+I+I_{\ast }%
\text{ and }\Psi =\left\vert \mathbf{g}\right\vert \sqrt{\left\vert \mathbf{g%
}\right\vert ^{2}-\frac{4}{m}\Delta I}\text{,}  \label{est1}
\end{eqnarray}%
on the scattering cross section $\sigma $, or, equivalently, the bound%
\begin{equation}
\!0\leq B\left( \left\vert \mathbf{g}\right\vert ,\cos \theta ,I,I_{\ast
},I^{\prime },I_{\ast }^{\prime }\right) \leq CE\left( 1+\frac{1}{\Psi
^{1-\gamma /2}}\right) \mathbf{1}_{m\left\vert \mathbf{g}\right\vert
^{2}>4\Delta I}  \label{est1a}
\end{equation}%
on the collision kernel $\left( \ref{ck1}\right) $, for some positive
constant $C>0$. Then the following result may be obtained.

\begin{theorem}
\label{Thm1}Assume that the scattering cross section $\sigma $, satisfy the
bound $\left( \ref{est1}\right) $ for some positive number $\gamma $, such
that $0<\gamma <1$. Then the operator $K$ given by $\left( \ref{dec1}\right) 
$ is a self-adjoint compact operator on $L^{2}\left( d\boldsymbol{\xi \,}%
\mathbf{\,}dI\right) $.
\end{theorem}

The proof of Theorem \ref{Thm1} will be addressed in Section $\ref{PT1}$.

\begin{corollary}
\label{Cor1}The linearized collision operator $\mathcal{L}$, with scattering
cross section satisfying $\left( \ref{est1}\right) $, is a closed, densely
defined, self-adjoint operator on $L^{2}\left( d\boldsymbol{\xi \,}\mathbf{\,%
}dI\right) $.
\end{corollary}

\begin{proof}
The multiplication operator $\Lambda $, where $\Lambda f=vf$, is a closed,
densely defined, self-adjoint operator. Hence, by Theorem \ref{Thm1}, $%
\mathcal{L}=\Lambda -K$ is closed, as the sum of a closed and a bounded
operator, densely defined, since the domains of the linear operators $%
\mathcal{L}$ and $\Lambda $ are equal, $D(\mathcal{L})=D(\Lambda )$, and
self-adjoint, since the set of self-adjoint operators is closed under
addition of bounded self-adjoint operators, see Theorem 4.3 of Chapter V in 
\cite{Kato}.
\end{proof}

\begin{remark}
The collision kernels (cf. Model 1-3 in \cite{GP-20})

1)%
\begin{equation*}
B=bE^{\alpha /2}
\end{equation*}

2)%
\begin{equation*}
B=b\left( R^{\alpha /2}\left\vert \mathbf{g}\right\vert ^{\alpha
}+(1-R\right) ^{\alpha /2}\left( \frac{I+I_{\ast }}{m}\right) ^{\alpha
/2})\leq \frac{2^{\alpha }+1}{m^{\alpha /2}}bE^{\alpha /2}
\end{equation*}

3)%
\begin{eqnarray*}
B&=&b\left( R^{\alpha /2}\left\vert \mathbf{g}\right\vert ^{\alpha }+\left(
r\left( 1-R\right) \frac{I}{m}\right) ^{\alpha /2}+\left( \left( 1-r\right)
\left( 1-R\right) \frac{I_{\ast }}{m}\right) ^{\alpha /2}\right) \\
&\leq& \frac{%
2^{\alpha }+2}{m^{\alpha /2}}bE^{\alpha /2}
\end{eqnarray*}%
where $\left\{ R=\dfrac{m\left\vert \mathbf{g}^{\prime }\right\vert ^{2}}{4E}%
,r=\dfrac{I^{\prime }}{(1-R)E}\right\} \in \left[ 0,1\right] ^{2}$, satisfy
the bound $\left( \ref{est1a}\right) $ for positive numbers $\alpha $ less
or equal to $2$, $0<\alpha \leq 2$, and bounded functions $b=b\left( \cos
\theta \right) $. Indeed, $E^{\alpha /2}$ is bounded above by $E$ for $%
\alpha =2$, or, if $E\geq 1$, while for $\alpha \in \left( 0,2\right) $, $%
E^{\alpha /2}=\dfrac{E}{E^{1-\alpha /2}}<\dfrac{E}{E^{1-\gamma /2}}\leq 
\dfrac{E}{\Psi ^{1-\gamma /2}}$ for $\gamma =\dfrac{\alpha }{2}\in \left(
0,1\right) $ if $E<1$.
\end{remark}

\begin{remark}
By slight modifications in the proof of Theorem \ref{Thm1} in Section $\ref%
{PT1}$\ one may replace the bounds $\left( \ref{est1}\right) $ and $\left( %
\ref{est1a}\right) $, with%
\begin{equation*}
0\leq \sigma \mathbf{1}_{m\left\vert \mathbf{g}\right\vert ^{2}>4\Delta
I}\leq C\frac{\Psi +\Psi ^{\gamma /2}}{\left\vert \mathbf{g}\right\vert ^{2}}%
\frac{1+E^{\eta -1}}{E^{\delta -1/2}}\left( I^{\prime }I_{\ast }^{\prime
}\right) ^{\delta /2-1}\mathbf{1}_{m\left\vert \mathbf{g}\right\vert
^{2}>4\Delta I}
\end{equation*}%
and%
\begin{equation*}
\!0\leq B\leq C\left( E^{\eta }+E\right) \left( 1+\frac{1}{\Psi ^{1-\gamma
/2}}\right) \mathbf{1}_{m\left\vert \mathbf{g}\right\vert ^{2}>4\Delta I}%
\text{,}
\end{equation*}%
for any $1<\eta <\dfrac{\delta +1}{2}$, respectively.
\end{remark}

Now consider the scattering cross section%
\begin{equation}
\sigma =C\dfrac{\sqrt{\left\vert \mathbf{g}\right\vert ^{2}-\frac{4}{m}%
\Delta I}}{\left\vert \mathbf{g}\right\vert E^{\delta +\left( \alpha
-1\right) /2}}\left( I^{\prime }I_{\ast }^{\prime }\right) ^{\delta /2-1}%
\text{ if }\left\vert \mathbf{g}\right\vert ^{2}>\frac{4}{m}\Delta I
\label{e1}
\end{equation}%
or, equivalently, the collision kernel $\left( \ref{ck1}\right) $%
\begin{equation}
\!B=CE^{1-\alpha /2}1_{E>0}  \label{e1a}
\end{equation}%
for some positive constant $C>0$ and nonnegative number $\alpha $ less than $%
2$, $0\leq \alpha <2$ - cf. hard sphere models for $\alpha =1$.

In fact, it would be enough with the bounds, if $\left\vert \mathbf{g}%
\right\vert ^{2}>\frac{4}{m}\Delta I$, 
\begin{equation}
C_{-}\dfrac{\sqrt{\left\vert \mathbf{g}\right\vert ^{2}-\frac{4}{m}\Delta I}%
}{\left\vert \mathbf{g}\right\vert E^{\delta +\left( \alpha -1\right) /2}}%
\left( I^{\prime }I_{\ast }^{\prime }\right) ^{\delta /2-1}\leq \sigma \leq
C_{+}\dfrac{\sqrt{\left\vert \mathbf{g}\right\vert ^{2}-\frac{4}{m}\Delta I}%
}{\left\vert \mathbf{g}\right\vert E^{\delta +\left( \alpha -1\right) /2}}%
\left( I^{\prime }I_{\ast }^{\prime }\right) ^{\delta /2-1}  \label{ie1}
\end{equation}%
on the scattering cross sections, or, equivalently, the bounds%
\begin{equation}
\!C_{-}E^{1-\alpha /2}1_{E>0}\leq B\leq C_{+}E^{1-\alpha /2}1_{E>0}
\label{ie1a}
\end{equation}%
on the collision kernel $\left( \ref{ck1}\right) $, for some positive
constants $C_{\pm }>0$ and nonnegative number $\alpha $ less than $2$, $%
0\leq \alpha <2$ - cf. hard potential with cut-off models, with "super hard"
potentials for $0\leq \alpha <1$.

\begin{theorem}
\label{Thm2} The linearized collision operator $\mathcal{L}$, with
scattering cross section $\left( \ref{e1}\right) $ (or $\left( \ref{ie1}%
\right) $), can be split into a positive multiplication operator $\Lambda $,
where $\Lambda f=vf$, with $\nu =\nu (\left\vert \boldsymbol{\xi }%
\right\vert ,I)$, minus a compact operator $K$ on $L^{2}\left( d\boldsymbol{%
\xi \,}\mathbf{\,}dI\right) $%
\begin{equation}
\mathcal{L}=\Lambda -K,  \label{dec3}
\end{equation}%
such that there exist positive numbers $\nu _{-}$ and $\nu _{+}$, $0<\nu
_{-}<\nu _{+}$, such that for any positive number $\varepsilon >0$ 
\begin{equation}
\nu _{-}\left( 1+\left\vert \boldsymbol{\xi }\right\vert +\sqrt{I}\right)
^{2-\alpha }\leq \nu (\left\vert \boldsymbol{\xi }\right\vert ,I)\leq \nu
_{+}\left( 1+\left\vert \boldsymbol{\xi }\right\vert +\sqrt{I}\right)
^{2-\alpha +\varepsilon }\text{ for all }\boldsymbol{\xi }\in \mathbb{R}^{3}%
\text{.}  \label{ine1}
\end{equation}
\end{theorem}

The decomposition $\left( \ref{dec3}\right) $ follows by decomposition $%
\left( \ref{dec2}\right) ,\left( \ref{dec1}\right) $ and Theorem \ref{Thm1},
while the bounds $\left( \ref{ine1}\right) $ on the collision frequency will
be proven in Section $\ref{PT2}$.

\begin{corollary}
\label{Cor2}The linearized collision operator $\mathcal{L}$, with scattering
cross section $\left( \ref{e1}\right) $ (or $\left( \ref{ie1}\right) $), is
a Fredholm operator.
\end{corollary}

\begin{proof}
By Theorem \ref{Thm2} the multiplication operator $\Lambda $ is coercive,
why it is a Fredholm operator. Furthermore, the set of Fredholm operators is
closed under addition of compact operators, see Theorem 5.26 of Chapter IV
in \cite{Kato} and its proof, so, by Theorem \ref{Thm2}, $\mathcal{L}$ is a
Fredholm operator.
\end{proof}

For hard sphere like models and "super hard" potential like models, we
obtain the following result.

\begin{corollary}
\label{Cor3}For the linearized collision operator $\mathcal{L}$, with
scattering cross section $\left( \ref{e1}\right) $ (or $\left( \ref{ie1}%
\right) $) with $0\leq \alpha \leq 1$, there exists a positive number $%
\lambda $, $0<\lambda <1$, such that 
\begin{equation*}
\left( h,\mathcal{L}h\right) \geq \lambda \left( h,\nu (\left\vert 
\boldsymbol{\xi }\right\vert ,I)h\right) \geq \lambda \nu _{-}\left(
h,\left( 1+\left\vert \boldsymbol{\xi }\right\vert \right) h\right)
\end{equation*}%
for all $h\in D\left( \mathcal{L}\right) \cap \mathrm{Im}\mathcal{L}$.
\end{corollary}

\begin{proof}
Let $h\in D\left( \mathcal{L}\right) \cap \left( \mathrm{ker}\mathcal{L}%
\right) ^{\perp }=D\left( \mathcal{L}\right) \cap \mathrm{Im}\mathcal{L}$.
As a Fredholm operator, $\mathcal{L}$ is closed with a closed range, and as
a compact operator, $K$ is bounded, and so there are positive constants $\nu
_{0}>0$ and $c_{K}>0$, such that%
\begin{equation*}
(h,Lh)\geq \nu _{0}(h,h)\text{ and }(h,Kh)\leq c_{K}(h,h).
\end{equation*}%
Let $\lambda =\dfrac{\nu _{0}}{\nu _{0}+c_{K}}$. Then the corollary follows,
since 
\begin{eqnarray*}
(h,Lh) &=&(1-\lambda )(h,\mathcal{L}h)+\lambda (h,(\nu (|\boldsymbol{\xi }|,I)-K)h) \\
&\geq &(1-\lambda )\nu _{0}(h,h)+\lambda (h,\nu (|\boldsymbol{\xi }%
|,I)h)-\lambda c_{K}(h,h) \\
&=&(\nu _{0}-\lambda (\nu _{0}+c_{K}))(h,h)+\lambda (h,\nu (|\boldsymbol{\xi 
}|,I)h) \\
&=&\lambda (h,\nu (|\boldsymbol{\xi }|,I)h)\text{.}
\end{eqnarray*}
\end{proof}

\begin{remark}
For hard sphere like models, as well as, "super hard" potential like models,
the linearized collision operator satisfies all the properties of the
general linear operator in the abstract half-space problem considered in 
\cite{Be-21}, under the assumption $B=\xi _{x}$ for the linear operator $B$
in \cite{Be-21}, by Proposition \ref{Prop1} and Corollaries \ref{Cor1}-\ref%
{Cor3}. Indeed, by Proposition \ref{Prop1} and Corollaries \ref{Cor1}-\ref%
{Cor3}, with $0\leq \alpha \leq 1$, in the expressions $\left( \ref{e1}%
\right) $ and $\left( \ref{e1a}\right) $, or, the bounds $\left( \ref{ie1}%
\right) $ and $\left( \ref{ie1a}\right) $, for the scattering cross section
and collision kernel, respectively, the linearized collision operator will
be a nonnegative self-adjoint Fredholm operator on the real Hilbert space $%
\mathcal{\mathfrak{h}}=L^{2}\left( d\boldsymbol{\xi \,}\mathbf{\,}dI\right) $%
, and moreover, there exists a positive number $\mu $, $\mu >0$, such that 
\begin{equation*}
\left( h,\mathcal{L}h\right) \geq \mu \left( h,\left( 1+\left\vert \xi
_{x}\right\vert \right) h\right) \text{ for all }h\in D\left( \mathcal{L}%
\right) \cap \mathrm{Im}\mathcal{L}.
\end{equation*}
\end{remark}

\section{\label{PT1}Compactness}

This section concerns the proof of Theorem \ref{Thm1}. Note that in the
proof the kernels are rewritten in such a way that $\boldsymbol{\xi }_{\ast
} $ and $I_{\ast }$ - and not $\boldsymbol{\xi }^{\prime }$ and $I^{\prime }$%
, or, $\boldsymbol{\xi }_{\ast }^{\prime }$ and $I_{\ast }^{\prime }$ -
always will be the arguments of the distribution functions. Then there will
be essentially two different types of kernels; either $\left( \boldsymbol{%
\xi }_{\ast },I_{\ast }\right) $ are arguments in the loss term (like $%
\left( \boldsymbol{\xi },I\right) $) or in the gain term (unlike $\left( 
\boldsymbol{\xi },I\right) $) of the collision operator. The kernel of the
term from the loss part of the collision operator will be shown\ to be
Hilbert-Schmidt in a quite direct way, while the kernel of the terms from
the gain part of the collision operator will be shown to be approximately
Hilbert-Schmidt, in the sense of Lemma \ref{LGD} below, or, equivalently, a
uniform limit of Hilbert-Schmidt integral operators. In fact, with similar
arguments as in the proof below one may show that if the number of internal
degrees of freedom $\delta $ is greater than $5/2$, $\delta >5/2$, then,
under the assumption $\left( \ref{est1}\right) $, the terms from the gain
part are Hilbert-Schmidt integral operators as well.

Denote, for any (nonzero) natural number $N$, 
\begin{align*}
\mathfrak{h}_{N}:=& \left\{ (\boldsymbol{\xi },\boldsymbol{\xi }_{\ast
},I,I_{\ast })\in \left( \mathbb{R}^{3}\times \mathbb{R}_{+}\right)
^{2}:\left\vert \boldsymbol{\xi }-\boldsymbol{\xi }_{\ast }\right\vert \geq 
\frac{1}{N}\text{; }\left\vert \boldsymbol{\xi }\right\vert \leq N\right\} 
\text{, and} \\
b^{(N)}=& b^{(N)}(\boldsymbol{\xi },\boldsymbol{\xi }_{\ast },I,I_{\ast
}):=b(\boldsymbol{\xi },\boldsymbol{\xi }_{\ast },I,I_{\ast })\mathbf{1}_{%
\mathfrak{h}_{N}}\text{.}
\end{align*}%
Then we have the following lemma, cf Glassey \cite[Lemma 3.5.1]{Glassey} and
Drange \cite{Dr-75}.

\begin{lemma}
\label{LGD} Assume that $Tf\left( \boldsymbol{\xi },I\right) =\int_{\mathbb{R%
}^{3}\times \mathbb{R}_{+}}b(\boldsymbol{\xi },\boldsymbol{\xi }_{\ast
},I,I_{\ast })f\left( \boldsymbol{\xi }_{\ast },I_{\ast }\right) \,d%
\boldsymbol{\xi }_{\ast }dI_{\ast }$, with $b(\boldsymbol{\xi },\boldsymbol{%
\xi }_{\ast },I,I_{\ast })\geq 0$. Then $T$ is compact on $L^{2}\left( d%
\boldsymbol{\xi \,}dI\right) $ if

(i) $\int_{\mathbb{R}^{3}\times \mathbb{R}_{+}}b(\boldsymbol{\xi },%
\boldsymbol{\xi }_{\ast },I,I_{\ast })\,d\boldsymbol{\xi }$ $dI$ is bounded
in $\left( \boldsymbol{\xi }_{\ast },I_{\ast }\right) $;

(ii) $b^{(N)}\in L^{2}\left( d\boldsymbol{\xi \,}d\boldsymbol{\xi }_{\ast }dI%
\boldsymbol{\,}dI_{\ast }\right) $ for any (nonzero) natural number $N$;

(iii) $\underset{\in \mathbb{R}^{3}\times \mathbb{R}_{+}}{\underset{\left( 
\boldsymbol{\xi },I\right) }{\sup }}\int_{\mathbb{R}^{3}\times \mathbb{R}%
_{+}}b(\boldsymbol{\xi },\boldsymbol{\xi }_{\ast },I,I_{\ast })-b^{(N)}(%
\boldsymbol{\xi },\boldsymbol{\xi }_{\ast },I,I_{\ast })\,d\boldsymbol{\xi }%
_{\ast }dI_{\ast }\rightarrow 0$ as $N\rightarrow \infty $.
\end{lemma}

Then $T$ is the uniform limit of Hilbert-Schmidt integral operators \cite[%
Lemma 3.5.1]{Glassey}, and we say that the kernel $b(\boldsymbol{\xi },%
\boldsymbol{\xi }_{\ast },I,I_{\ast })$ is approximately Hilbert-Schmidt,
while $T$ is an approximately Hilbert-Schmidt integral operator. Note that
by this definition a Hilbert-Schmidt integral operator, will also be an
approximately Hilbert-Schmidt integral operator. The reader is referred to 
\cite[Lemma 3.5.1]{Glassey} for a proof of Lemma \ref{LGD}.

Note that throughout the proof, $C$ will denote a generic positive constant.

\begin{proof}
Rewrite expression $\left( \ref{dec1}\right) $ as 
\begin{eqnarray*}
Kh &=&M^{-1/2}\int_{\left( \mathbb{R}^{3}\times \mathbb{R}_{+}\right) ^{3}}w(%
\boldsymbol{\xi },\boldsymbol{\xi }_{\ast },I,I_{\ast }\left\vert 
\boldsymbol{\xi }^{\prime },\boldsymbol{\xi }_{\ast }^{\prime },I^{\prime
},I_{\ast }^{\prime }\right. ) \\
&&\times \left( \frac{h_{\ast }}{M_{\ast }^{1/2}}-\frac{h^{\prime }}{\left(
M^{\prime }\right) ^{1/2}}-\frac{h_{\ast }^{\prime }}{\left( M_{\ast
}^{\prime }\right) ^{1/2}}\right) \,d\boldsymbol{\xi }_{\ast }d\boldsymbol{%
\xi }^{\prime }d\boldsymbol{\xi }_{\ast }^{\prime }dI_{\ast }dI^{\prime
}dI_{\ast }^{\prime }\text{,}
\end{eqnarray*}%
with 
\begin{equation*}
w(\boldsymbol{\xi },\boldsymbol{\xi }_{\ast },I,I_{\ast }\left\vert 
\boldsymbol{\xi }^{\prime },\boldsymbol{\xi }_{\ast }^{\prime },I^{\prime
},I_{\ast }^{\prime }\right. )=\frac{\left( MM_{\ast }M^{\prime }M_{\ast
}^{\prime }\right) ^{1/2}}{\left( II_{\ast }I^{\prime }I_{\ast }^{\prime
}\right) ^{\delta /4-1/2}}W(\boldsymbol{\xi },\boldsymbol{\xi }_{\ast
},I,I_{\ast }\left\vert \boldsymbol{\xi }^{\prime },\boldsymbol{\xi }_{\ast
}^{\prime },I^{\prime },I_{\ast }^{\prime }\right. )\text{.}
\end{equation*}%
Due to relations $\left( \ref{rel1}\right) $, the relations%
\begin{eqnarray}
w(\boldsymbol{\xi },\boldsymbol{\xi }_{\ast },I,I_{\ast }\left\vert 
\boldsymbol{\xi }^{\prime },\boldsymbol{\xi }_{\ast }^{\prime },I^{\prime
},I_{\ast }^{\prime }\right. ) &=&w(\boldsymbol{\xi }_{\ast },\boldsymbol{%
\xi },I_{\ast },I\left\vert \boldsymbol{\xi }_{\ast }^{\prime },\boldsymbol{%
\xi }^{\prime },I_{\ast }^{\prime },I^{\prime }\right. )  \notag \\
w(\boldsymbol{\xi },\boldsymbol{\xi }_{\ast },I,I_{\ast }\left\vert 
\boldsymbol{\xi }^{\prime },\boldsymbol{\xi }_{\ast }^{\prime },I^{\prime
},I_{\ast }^{\prime }\right. ) &=&w(\boldsymbol{\xi }^{\prime },\boldsymbol{%
\xi }_{\ast }^{\prime },I^{\prime },I_{\ast }^{\prime }\left\vert 
\boldsymbol{\xi },\boldsymbol{\xi }_{\ast },I,I_{\ast }\right. )  \notag \\
w(\boldsymbol{\xi },\boldsymbol{\xi }_{\ast },I,I_{\ast }\left\vert 
\boldsymbol{\xi }^{\prime },\boldsymbol{\xi }_{\ast }^{\prime },I^{\prime
},I_{\ast }^{\prime }\right. ) &=&w(\boldsymbol{\xi },\boldsymbol{\xi }%
_{\ast },I,I_{\ast }\left\vert \boldsymbol{\xi }_{\ast }^{\prime },%
\boldsymbol{\xi }^{\prime },I_{\ast }^{\prime },I^{\prime }\right. )
\label{rel2}
\end{eqnarray}%
are satisfied.

By first renaming $\left\{ \boldsymbol{\xi }_{\ast },I_{\ast }\right\}
\leftrightarrows \left\{ \boldsymbol{\xi }^{\prime },I^{\prime }\right\} $,
then $\left\{ \boldsymbol{\xi }_{\ast },I_{\ast }\right\} \leftrightarrows
\left\{ \boldsymbol{\xi }_{\ast }^{\prime },I_{\ast }^{\prime }\right\} $,
followed by applying the last relation in $\left( \ref{rel2}\right) $, 
\begin{eqnarray*}
&&\int_{\left( \mathbb{R}^{3}\times \mathbb{R}_{+}\right) ^{3}}w(\boldsymbol{%
\xi },\boldsymbol{\xi }_{\ast },I,I_{\ast }\left\vert \boldsymbol{\xi }%
^{\prime },\boldsymbol{\xi }_{\ast }^{\prime },I^{\prime },I_{\ast }^{\prime
}\right. )\,\frac{h_{\ast }^{\prime }}{\left( M_{\ast }^{\prime }\right)
^{1/2}}\,d\boldsymbol{\xi }_{\ast }d\boldsymbol{\xi }^{\prime }d\boldsymbol{%
\xi }_{\ast }^{\prime }dI_{\ast }dI^{\prime }dI_{\ast }^{\prime } \\
&=&\int_{\left( \mathbb{R}^{3}\times \mathbb{R}_{+}\right) ^{3}}w(%
\boldsymbol{\xi },\boldsymbol{\xi }^{\prime },I,I^{\prime }\left\vert 
\boldsymbol{\xi }_{\ast },\boldsymbol{\xi }_{\ast }^{\prime },I_{\ast
},I_{\ast }^{\prime }\right. )\,\frac{h_{\ast }^{\prime }}{\left( M_{\ast
}^{\prime }\right) ^{1/2}}\,d\boldsymbol{\xi }_{\ast }d\boldsymbol{\xi }%
^{\prime }d\boldsymbol{\xi }_{\ast }^{\prime }dI_{\ast }dI^{\prime }dI_{\ast
}^{\prime } \\
&=&\int_{\left( \mathbb{R}^{3}\times \mathbb{R}_{+}\right) ^{3}}w(%
\boldsymbol{\xi },\boldsymbol{\xi }^{\prime },I,I^{\prime }\left\vert 
\boldsymbol{\xi }_{\ast }^{\prime },\boldsymbol{\xi }_{\ast },I_{\ast
}^{\prime },I_{\ast }\right. )\,\frac{h_{\ast }}{M_{\ast }^{1/2}}\,d%
\boldsymbol{\xi }_{\ast }d\boldsymbol{\xi }^{\prime }d\boldsymbol{\xi }%
_{\ast }^{\prime }dI_{\ast }dI^{\prime }dI_{\ast }^{\prime } \\
&=&\int_{\left( \mathbb{R}^{3}\times \mathbb{R}_{+}\right) ^{3}}w(%
\boldsymbol{\xi },\boldsymbol{\xi }^{\prime },I,I^{\prime }\left\vert 
\boldsymbol{\xi }_{\ast },\boldsymbol{\xi }_{\ast }^{\prime },I_{\ast
},I_{\ast }^{\prime }\right. )\,\frac{h_{\ast }}{M_{\ast }^{1/2}}\,d%
\boldsymbol{\xi }_{\ast }d\boldsymbol{\xi }^{\prime }d\boldsymbol{\xi }%
_{\ast }^{\prime }dI_{\ast }dI^{\prime }dI_{\ast }^{\prime }\text{.}
\end{eqnarray*}%
Moreover, by renaming $\left\{ \boldsymbol{\xi }_{\ast },I_{\ast }\right\}
\leftrightarrows \left\{ \boldsymbol{\xi }^{\prime },I^{\prime }\right\} $, 
\begin{eqnarray*}
&&\int_{\left( \mathbb{R}^{3}\times \mathbb{R}_{+}\right) ^{3}}w(\boldsymbol{%
\xi },\boldsymbol{\xi }_{\ast },I,I_{\ast }\left\vert \boldsymbol{\xi }%
^{\prime },\boldsymbol{\xi }_{\ast }^{\prime },I^{\prime },I_{\ast }^{\prime
}\right. )\,\frac{h^{\prime }}{\left( M^{\prime }\right) ^{1/2}}\,d%
\boldsymbol{\xi }_{\ast }d\boldsymbol{\xi }^{\prime }d\boldsymbol{\xi }%
_{\ast }^{\prime }dI_{\ast }dI^{\prime }dI_{\ast }^{\prime } \\
&=&\int_{\left( \mathbb{R}^{3}\times \mathbb{R}_{+}\right) ^{3}}w(%
\boldsymbol{\xi },\boldsymbol{\xi }^{\prime },I,I^{\prime }\left\vert 
\boldsymbol{\xi }_{\ast },\boldsymbol{\xi }_{\ast }^{\prime },I_{\ast
},I_{\ast }^{\prime }\right. )\,\frac{h_{\ast }}{M_{\ast }^{1/2}}\,d%
\boldsymbol{\xi }_{\ast }d\boldsymbol{\xi }^{\prime }d\boldsymbol{\xi }%
_{\ast }^{\prime }dI_{\ast }dI^{\prime }dI_{\ast }^{\prime }\text{.}
\end{eqnarray*}%
It follows that%
\begin{eqnarray}
K\left( h\right) &=&\int_{\mathbb{R}^{3}\times \mathbb{R}_{+}}k(\boldsymbol{%
\xi },\boldsymbol{\xi }_{\ast },I,I_{\ast })\,h_{\ast }\,d\boldsymbol{\xi }%
_{\ast }dI_{\ast }\text{, where }  \notag \\
k(\boldsymbol{\xi },\boldsymbol{\xi }_{\ast },I,I_{\ast }) &=&k_{2}(%
\boldsymbol{\xi },\boldsymbol{\xi }_{\ast },I,I_{\ast })-k_{1}(\boldsymbol{%
\xi },\boldsymbol{\xi }_{\ast },I,I_{\ast })\text{,}  \notag
\end{eqnarray}%
with 
\begin{eqnarray}
&&k_{1}(\boldsymbol{\xi },\boldsymbol{\xi }_{\ast },I,I_{\ast })  \notag \\
&=&\left( MM_{\ast }\right) ^{-1/2}\int_{\left( \mathbb{R}^{3}\times \mathbb{%
R}_{+}\right) ^{2}}w(\boldsymbol{\xi },\boldsymbol{\xi }_{\ast },I,I_{\ast
}\left\vert \boldsymbol{\xi }^{\prime },\boldsymbol{\xi }_{\ast }^{\prime
},I^{\prime },I_{\ast }^{\prime }\right. )\,d\boldsymbol{\xi }^{\prime }d%
\boldsymbol{\xi }_{\ast }^{\prime }dI^{\prime }dI_{\ast }^{\prime }\text{ and%
}  \notag \\
&&k_{2}(\boldsymbol{\xi },\boldsymbol{\xi }_{\ast },I,I_{\ast })  \notag \\
&=&2\left( MM_{\ast }\right) ^{-1/2}\int_{\left( \mathbb{R}^{3}\times 
\mathbb{R}_{+}\right) ^{2}}w(\boldsymbol{\xi },\boldsymbol{\xi }^{\prime
},I,I^{\prime }\left\vert \boldsymbol{\xi }_{\ast },\boldsymbol{\xi }_{\ast
}^{\prime },I_{\ast },I_{\ast }^{\prime }\right. )\,d\boldsymbol{\xi }%
^{\prime }d\boldsymbol{\xi }_{\ast }^{\prime }dI^{\prime }dI_{\ast }^{\prime
}\text{.}  \label{k1}
\end{eqnarray}

Note that%
\begin{equation*}
k(\boldsymbol{\xi },\boldsymbol{\xi }_{\ast },I,I_{\ast })=k_{2}(\boldsymbol{%
\xi }_{\ast },\boldsymbol{\xi },I_{\ast },I)-k_{1}(\boldsymbol{\xi }_{\ast },%
\boldsymbol{\xi },I_{\ast },I)=k(\boldsymbol{\xi }_{\ast },\boldsymbol{\xi }%
,I_{\ast },I),
\end{equation*}%
since, by applying the first and the last relation in $\left( \ref{rel2}%
\right) $, 
\begin{eqnarray}
&&k_{1}(\boldsymbol{\xi },\boldsymbol{\xi }_{\ast },I,I_{\ast })  \notag \\
&=&\left( MM_{\ast }\right) ^{-1/2}\int_{\left( \mathbb{R}^{3}\times \mathbb{%
R}_{+}\right) ^{2}}w(\boldsymbol{\xi }_{\ast },\boldsymbol{\xi },I_{\ast
},I\left\vert \boldsymbol{\xi }_{\ast }^{\prime },\boldsymbol{\xi }^{\prime
},I_{\ast }^{\prime },I^{\prime }\right. )\,d\boldsymbol{\xi }^{\prime }d%
\boldsymbol{\xi }_{\ast }^{\prime }dI^{\prime }dI_{\ast }^{\prime }  \notag
\\
&=&\left( MM_{\ast }\right) ^{-1/2}\int_{\left( \mathbb{R}^{3}\times \mathbb{%
R}_{+}\right) ^{2}}w(\boldsymbol{\xi }_{\ast },\boldsymbol{\xi },I_{\ast
},I\left\vert \boldsymbol{\xi }^{\prime },\boldsymbol{\xi }_{\ast }^{\prime
},I^{\prime },I_{\ast }^{\prime }\right. )\,d\boldsymbol{\xi }^{\prime }d%
\boldsymbol{\xi }_{\ast }^{\prime }dI^{\prime }dI_{\ast }^{\prime }  \notag
\\
&=&k_{1}(\boldsymbol{\xi }_{\ast },\boldsymbol{\xi },I_{\ast },I)
\label{sa1}
\end{eqnarray}%
and, by applying the second relation in $\left( \ref{rel2}\right) $ and
renaming $\left\{ \boldsymbol{\xi }^{\prime },I^{\prime }\right\}
\leftrightarrows \left\{ \boldsymbol{\xi }_{\ast }^{\prime },I_{\ast
}^{\prime }\right\} $, 
\begin{eqnarray}
&&k_{2}(\boldsymbol{\xi },\boldsymbol{\xi }_{\ast },I,I_{\ast })  \notag \\
&=&2\left( MM_{\ast }\right) ^{-1/2}\int_{\left( \mathbb{R}^{3}\times 
\mathbb{R}_{+}\right) ^{2}}w(\boldsymbol{\xi }_{\ast },\boldsymbol{\xi }%
_{\ast }^{\prime },I_{\ast },I_{\ast }^{\prime }\left\vert \boldsymbol{\xi },%
\boldsymbol{\xi }^{\prime },I,I^{\prime }\right. )\,d\boldsymbol{\xi }%
^{\prime }d\boldsymbol{\xi }_{\ast }^{\prime }dI^{\prime }dI_{\ast }^{\prime
}  \notag \\
&=&2\left( MM_{\ast }\right) ^{-1/2}\int_{\left( \mathbb{R}^{3}\times 
\mathbb{R}_{+}\right) ^{2}}w(\boldsymbol{\xi }_{\ast },\boldsymbol{\xi }%
^{\prime },I_{\ast },I^{\prime }\left\vert \boldsymbol{\xi },\boldsymbol{\xi 
}_{\ast }^{\prime },I,I_{\ast }^{\prime }\right. )\,d\boldsymbol{\xi }%
^{\prime }d\boldsymbol{\xi }_{\ast }^{\prime }dI^{\prime }dI_{\ast }^{\prime
}  \notag \\
&=&k_{2}(\boldsymbol{\xi }_{\ast },\boldsymbol{\xi },I_{\ast },I)\text{.}
\label{sa2}
\end{eqnarray}

We now continue by proving the compactness for the two different types of
collision kernel separately.

\begin{figure}[h]
\centering
\includegraphics[width=0.6\textwidth]{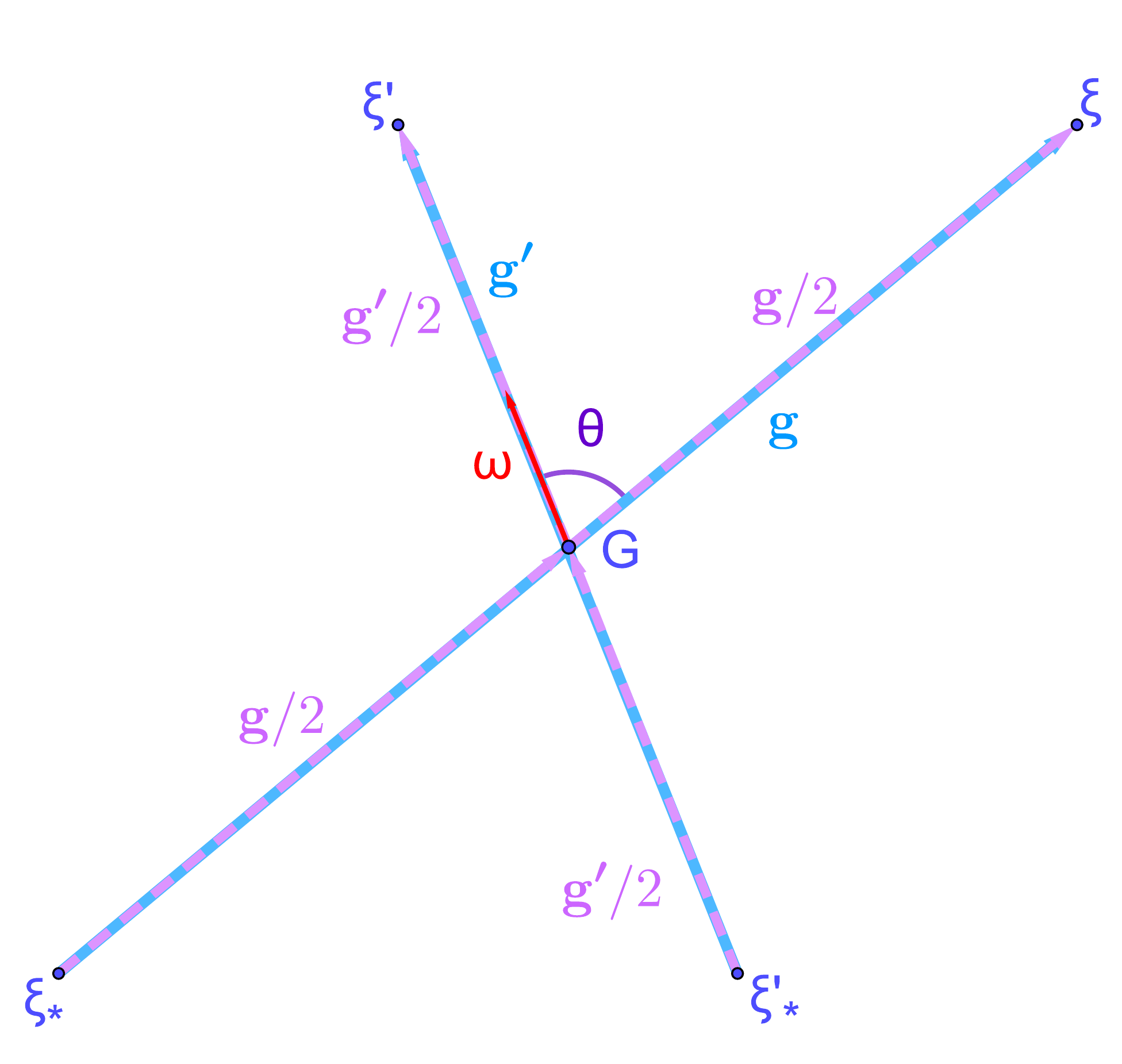}
\caption{Typical collision of $K_{1}$. Classical representation of an
inelastic collision.}
\label{fig1}
\end{figure}

\textbf{I. Compactness of }$K_{1}=\int_{\mathbb{R}^{3}\times \mathbb{R}%
_{+}}k_{1}(\boldsymbol{\xi },\boldsymbol{\xi }_{\ast },I,I_{\ast })\,h_{\ast
}\,d\boldsymbol{\xi }_{\ast }dI_{\ast }$.

A change of variables $\left\{ \boldsymbol{\xi }^{\prime },\boldsymbol{\xi }%
_{\ast }^{\prime }\right\} \rightarrow \left\{ \left\vert \mathbf{g}^{\prime
}\right\vert =\left\vert \boldsymbol{\xi }^{\prime }-\boldsymbol{\xi }_{\ast
}^{\prime }\right\vert ,\boldsymbol{\omega }=\dfrac{\mathbf{g}^{\prime }}{%
\left\vert \mathbf{g}^{\prime }\right\vert },\mathbf{G}^{\prime }=\dfrac{%
\boldsymbol{\xi }^{\prime }+\boldsymbol{\xi }_{\ast }^{\prime }}{2}\right\} $%
, cf. Figure $\ref{fig1}$, noting that $\left( \ref{df1}\right) $ and using
relation $\left( \ref{M1}\right) $, expression $\left( \ref{k1}\right) $ of $%
k_{1}$ may be transformed to 
\begin{eqnarray*}
&&k_{1}(\boldsymbol{\xi },\boldsymbol{\xi }_{\ast },I,I_{\ast }) \\
&=&\int_{\left( \mathbb{R}^{3}\times \mathbb{R}_{+}\right) ^{2}}\frac{\left(
M^{\prime }M_{\ast }^{\prime }\right) ^{1/2}}{\left( II_{\ast }I^{\prime
}I_{\ast }^{\prime }\right) ^{\delta /4-1/2}}W(\boldsymbol{\xi },\boldsymbol{%
\xi }_{\ast },I,I_{\ast }\left\vert \boldsymbol{\xi }^{\prime },\boldsymbol{%
\xi }_{\ast }^{\prime },I^{\prime },I_{\ast }^{\prime }\right. )\,d%
\boldsymbol{\xi }^{\prime }d\boldsymbol{\xi }_{\ast }^{\prime }dI^{\prime
}dI_{\ast }^{\prime } \\
&=&\int_{\mathbb{R}^{3}\times \mathbb{R}_{+}^{3}\times \mathbb{S}^{2}}\frac{%
\left( M^{\prime }M_{\ast }^{\prime }\right) ^{1/2}}{\left( II_{\ast
}I^{\prime }I_{\ast }^{\prime }\right) ^{\delta /4-1/2}}\left\vert \mathbf{g}%
^{\prime }\right\vert ^{2}W(\boldsymbol{\xi },\boldsymbol{\xi }_{\ast
},I,I_{\ast }\left\vert \boldsymbol{\xi }^{\prime },\boldsymbol{\xi }_{\ast
}^{\prime },I^{\prime },I_{\ast }^{\prime }\right. )\, \\
&&d\mathbf{G}^{\prime }d\left\vert \mathbf{g}^{\prime }\right\vert d%
\boldsymbol{\omega \,}dI^{\prime }dI_{\ast }^{\prime } \\
&=&\left( MM_{\ast }\right) ^{1/4}\left( II_{\ast }\right) ^{\delta
/8-1/4}\left\vert \mathbf{g}\right\vert \int_{\mathbb{S}^{2}\times \mathbb{R}%
_{+}^{2}}\frac{\left( M^{\prime }M_{\ast }^{\prime }\right) ^{1/4}}{\left(
I^{\prime }I_{\ast }^{\prime }\right) ^{\delta /8-1/4}}\sigma \mathbf{1}%
_{m\left\vert \mathbf{g}\right\vert ^{2}>4\Delta I}\,d\boldsymbol{\omega }%
\,dI^{\prime }dI_{\ast }^{\prime }\text{.}
\end{eqnarray*}%
Since, $E\geq \left( I^{\prime }I_{\ast }^{\prime }\right) ^{1/2}$ and $\Psi
\leq \left( EE^{\prime }\right) ^{1/2}=E$, it follows, by assumption $\left( %
\ref{est1}\right) $, that 
\begin{eqnarray*}
&&k_{1}^{2}(\boldsymbol{\xi },\boldsymbol{\xi }_{\ast },I,I_{\ast }) \\
&\leq &C\left( MM_{\ast }\right) ^{1/2}\frac{\left( II_{\ast }\right)
^{\delta /4-1/2}}{\left\vert \mathbf{g}\right\vert ^{2}} \\
&&\times \left( \int_{\mathbb{S}^{2}\times \mathbb{R}_{+}^{2}}\frac{\left(
I^{\prime }I_{\ast }^{\prime }\right) ^{\delta /2-1}}{E^{\delta -1/2}}%
e^{-I^{\prime }/4}e^{-I_{\ast }^{\prime }/4}\left( \Psi +\Psi ^{\gamma
/2}\right) ^{2}\mathbf{1}_{m\left\vert \mathbf{g}\right\vert ^{2}>4\Delta
I}\,d\boldsymbol{\omega }\,dI^{\prime }dI_{\ast }^{\prime }\right) ^{2} \\
&\leq &C\frac{\left( MM_{\ast }\right) ^{1/2}}{\left( II_{\ast }\right)
^{\delta /4-1/2}}\frac{\left( II_{\ast }\right) ^{\delta /2-1}}{\left\vert 
\mathbf{g}\right\vert ^{2}}\left( E+E^{\gamma /2}\right) ^{2}\left( \int_{%
\mathbb{S}^{2}}\,d\boldsymbol{\omega }\,\right) ^{2}\left( \int_{0}^{\infty }%
\frac{e^{-I/4}}{I^{3/4}}\,dI\right) ^{4}\text{.}
\end{eqnarray*}%
Now, noting that 
\begin{equation*}
m\frac{\left\vert \boldsymbol{\xi }\right\vert ^{2}}{2}+m\frac{\left\vert 
\boldsymbol{\xi }_{\ast }\right\vert ^{2}}{2}+I+I_{\ast }=m\left\vert 
\mathbf{G}\right\vert ^{2}+m\frac{\left\vert \mathbf{g}\right\vert ^{2}}{4}%
+I+I_{\ast }=m\left\vert \mathbf{G}\right\vert ^{2}+E\text{, \ }
\end{equation*}%
the bound 
\begin{equation}
k_{1}^{2}(\boldsymbol{\xi },\boldsymbol{\xi }_{\ast },I,I_{\ast })\leq
Ce^{-m\left\vert \mathbf{G}\right\vert ^{2}/2-E/2}\frac{\left( II_{\ast
}\right) ^{\delta /2-1}}{\left\vert \mathbf{g}\right\vert ^{2}}\left(
E+E^{\gamma /2}\right) ^{2}  \label{b1}
\end{equation}%
may be obtained. Then, by applying the bound $\left( \ref{b1}\right) $ and
first changing variables of integration $\left\{ \boldsymbol{\xi },%
\boldsymbol{\xi }_{\ast }\right\} \rightarrow \left\{ \mathbf{g},\mathbf{G}%
\right\} $, with unitary Jacobian, and then to spherical coordinates, 
\begin{eqnarray*}
&&\int_{\left( \mathbb{R}^{3}\times \mathbb{R}_{+}\right) ^{2}}k_{1}^{2}(%
\boldsymbol{\xi },\boldsymbol{\xi }_{\ast },I,I_{\ast })d\boldsymbol{\xi \,}d%
\boldsymbol{\xi }_{\ast }dI\boldsymbol{\,}dI_{\ast } \\
&\leq &C\int_{\left( \mathbb{R}^{3}\times \mathbb{R}_{+}\right)
^{2}}e^{-m\left\vert \mathbf{G}\right\vert ^{2}/2-E/2}\frac{\left( II_{\ast
}\right) ^{\delta /2-1}}{\left\vert \mathbf{g}\right\vert ^{2}}\left(
E+E^{\gamma /2}\right) ^{2}d\mathbf{g}\boldsymbol{\,}d\mathbf{G}\boldsymbol{%
\,}dI\boldsymbol{\,}dI_{\ast } \\
&=&C\int_{\mathbb{R}_{+}^{3}}e^{-m\left\vert \mathbf{g}\right\vert
^{2}/8}e^{-I/2}e^{-I_{\ast }/2}\left( II_{\ast }\right) ^{\delta /2-1}\left(
E+E^{\gamma /2}\right) ^{2}d\left\vert \mathbf{g}\right\vert \boldsymbol{\,}%
dI\boldsymbol{\,}dI_{\ast } \\
&&\times \int_{0}^{\infty }R^{2}e^{-R^{2}}dR\left( \int_{\mathbb{S}^{2}}\,d%
\boldsymbol{\omega }\,\right) ^{2} \\
&\leq &C\int_{\mathbb{R}_{+}^{3}}e^{-m\left\vert \mathbf{g}\right\vert
^{2}/8}e^{-I/2}e^{-I_{\ast }/2}\left( II_{\ast }\right) ^{\delta /2-1}\left(
1+E^{2}\right) \,d\left\vert \mathbf{g}\right\vert \boldsymbol{\,}dI%
\boldsymbol{\,}dI_{\ast } \\
&\leq &C\int_{0}^{\infty }e^{-m\left\vert \mathbf{g}\right\vert
^{2}/8}\left( 1+\frac{m^{2}}{16}\left\vert \mathbf{g}\right\vert ^{4}\right)
\,d\left\vert \mathbf{g}\right\vert \boldsymbol{\,}\left( \int_{0}^{\infty
}\left( 1+I\right) ^{2}e^{-I/2}I^{\delta /2-1}\,dI\right) ^{2} \\
&=&C
\end{eqnarray*}%
Hence,%
\begin{equation*}
K_{1}=\int_{\mathbb{R}^{3}\times \mathbb{R}_{+}}k_{1}(\boldsymbol{\xi },%
\boldsymbol{\xi }_{\ast },I,I_{\ast })\,h_{\ast }\,d\boldsymbol{\xi }_{\ast
}dI_{\ast }
\end{equation*}%
is a Hilbert-Schmidt integral operator and as such compact on $L^{2}\left( d%
\boldsymbol{\xi \,}dI\right) $, see e.g. Theorem 7.83 in \cite{RenardyRogers}%
.

\begin{figure}[h]
\centering
\includegraphics[width=0.6\textwidth]{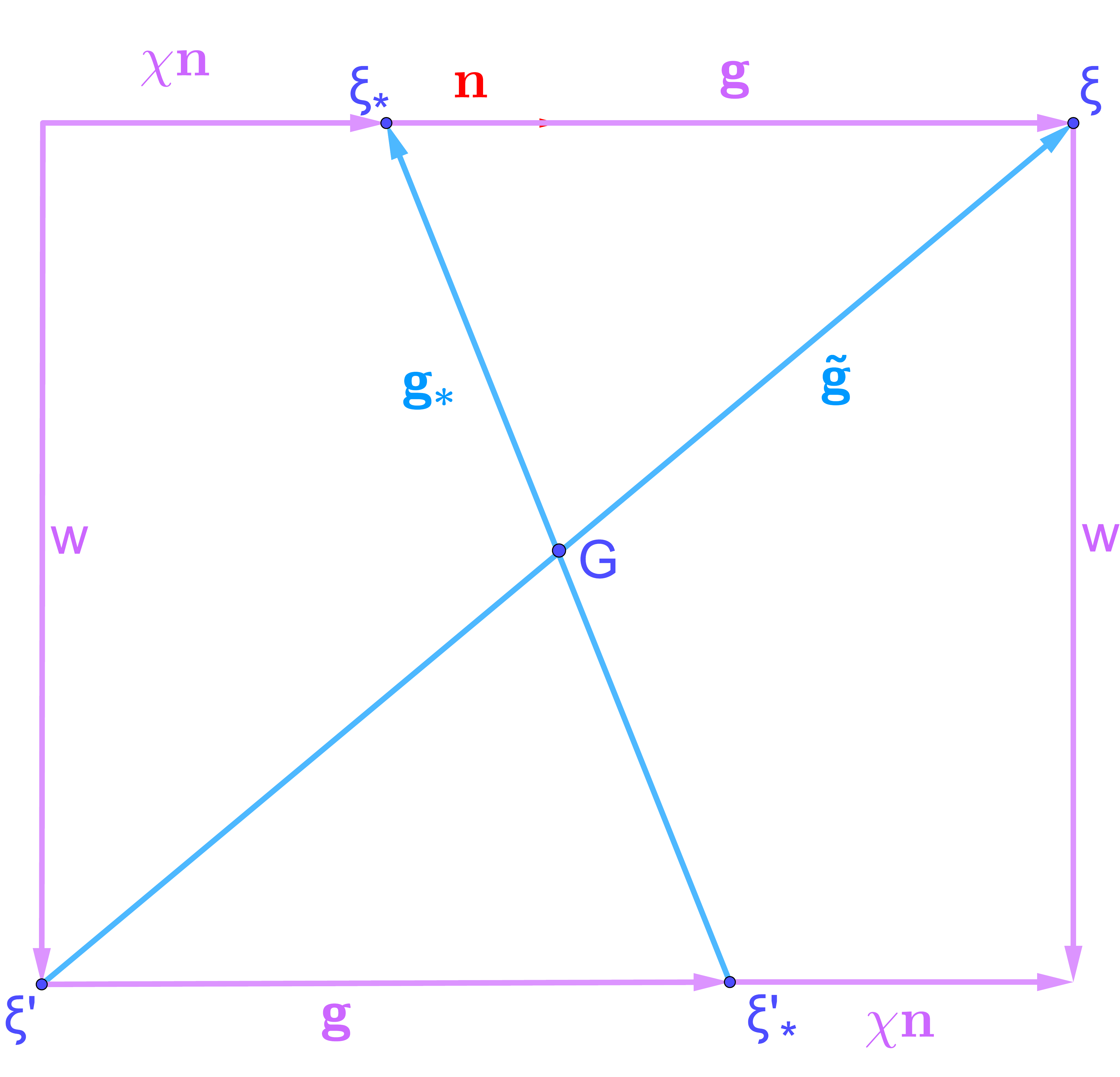}
\caption{Typical collision of $K_{2}$.}
\label{fig2}
\end{figure}

\textbf{II. Compactness of }$K_{2}=\int_{\mathbb{R}^{3}\times \mathbb{R}%
_{+}}k_{2}(\boldsymbol{\xi },\boldsymbol{\xi }_{\ast },I,I_{\ast })\,h_{\ast
}\,d\boldsymbol{\xi }_{\ast }dI_{\ast }$.

Denote, cf. Figure $\ref{fig2}$,%
\begin{eqnarray*}
\chi &=&\left( \boldsymbol{\xi }_{\ast }-\boldsymbol{\xi }^{\prime }\right)
\cdot \frac{\mathbf{g}}{\left\vert \mathbf{g}\right\vert }\text{, with }%
\mathbf{g}=\boldsymbol{\xi }-\boldsymbol{\xi }\mathbf{_{\ast }}\text{;} \\
\mathbf{g}^{\prime } &=&\boldsymbol{\xi }^{\prime }-\boldsymbol{\xi }_{\ast
}^{\prime },\text{ }\widetilde{\mathbf{g}}=\boldsymbol{\xi }-\boldsymbol{\xi 
}^{\prime }\text{, and }\mathbf{g}_{\ast }=\boldsymbol{\xi }_{\ast }-%
\boldsymbol{\xi }_{\ast }^{\prime }\text{.}
\end{eqnarray*}

Note that, see Figure $\ref{fig2}$,%
\begin{eqnarray*}
&&W(\boldsymbol{\xi },\boldsymbol{\xi }^{\prime },I,I^{\prime }\left\vert 
\boldsymbol{\xi }_{\ast },\boldsymbol{\xi }_{\ast }^{\prime },I_{\ast
},I_{\ast }^{\prime }\right. ) \\
&=&4m\left( II^{\prime }\right) ^{\delta /2-1}\widetilde{\sigma }\frac{%
\left\vert \widetilde{\mathbf{g}}\right\vert }{\left\vert \mathbf{g}_{\ast
}\right\vert }\delta _{1}\left( \frac{m}{2}\left( \left\vert \boldsymbol{\xi 
}\right\vert ^{2}-\left\vert \boldsymbol{\xi }_{\ast }\right\vert
^{2}+\left\vert \boldsymbol{\xi }^{\prime }\right\vert ^{2}-\left\vert 
\boldsymbol{\xi }_{\ast }^{\prime }\right\vert ^{2}\right) -\Delta I_{\ast
}\right) \\
&&\times \delta _{3}\left( \boldsymbol{\xi }^{\prime }-\boldsymbol{\xi }%
_{\ast }^{\prime }+\boldsymbol{\xi }-\boldsymbol{\xi }_{\ast }\right) \\
&=&4m\left( II^{\prime }\right) ^{\delta /2-1}\widetilde{\sigma }\frac{%
\left\vert \widetilde{\mathbf{g}}\right\vert }{\left\vert \mathbf{g}_{\ast
}\right\vert }\delta _{1}\left( m\left\vert \mathbf{g}\right\vert \chi
-\Delta I_{\ast }\right) \delta _{3}\left( \mathbf{g}^{\prime }+\mathbf{g}%
\right) \\
&=&4\left( II^{\prime }\right) ^{\delta /2-1}\widetilde{\sigma }\frac{%
\left\vert \widetilde{\mathbf{g}}\right\vert }{\left\vert \mathbf{g}_{\ast
}\right\vert \left\vert \mathbf{g}\right\vert }\delta _{1}\left( \chi -\frac{%
\Delta I_{\ast }}{m\left\vert \mathbf{g}\right\vert }\right) \delta
_{3}\left( \mathbf{g}^{\prime }+\mathbf{g}\right) \text{, with} \\
&&\Delta I_{\ast }=I_{\ast }+I_{\ast }^{\prime }-I-I^{\prime }\text{, and }%
\widetilde{\sigma }=\sigma \left( \left\vert \widetilde{\mathbf{g}}%
\right\vert ,\frac{\widetilde{\mathbf{g}}\cdot \mathbf{g}_{\ast }}{%
\left\vert \widetilde{\mathbf{g}}\right\vert \left\vert \mathbf{g}_{\ast
}\right\vert },I,I^{\prime },I_{\ast },I_{\ast }^{\prime }\right) \text{.}
\end{eqnarray*}

By a change of variables $\left\{ \boldsymbol{\xi }^{\prime },\boldsymbol{%
\xi }_{\ast }^{\prime }\right\} \rightarrow \left\{ \mathbf{g}^{\prime }=%
\boldsymbol{\xi }^{\prime }-\boldsymbol{\xi }_{\ast }^{\prime },~\mathbf{h}=%
\boldsymbol{\xi }^{\prime }-\boldsymbol{\xi }_{\ast }\right\} $, where 
\begin{equation*}
d\boldsymbol{\xi }^{\prime }d\boldsymbol{\xi }_{\ast }^{\prime }=d\mathbf{g}%
^{\prime }d\mathbf{h}=d\mathbf{g}^{\prime }d\chi d\mathbf{w}\text{,\ with }%
\mathbf{w}=\boldsymbol{\xi }^{\prime }-\boldsymbol{\xi }_{\ast }+\chi 
\mathbf{n}\text{ and }\mathbf{n}=\frac{\mathbf{g}}{\left\vert \mathbf{g}%
\right\vert }\text{,}
\end{equation*}%
the expression $\left( \ref{k1}\right) $ of $k_{2}$ may be transformed in
the following way%
\begin{eqnarray*}
&&k_{2}(\boldsymbol{\xi },\boldsymbol{\xi }_{\ast },I,I_{\ast }) \\
&=&\int_{\left( \mathbb{R}^{3}\times \mathbb{R}_{+}\right) ^{2}}2\frac{%
\left( M^{\prime }M_{\ast }^{\prime }\right) ^{1/2}}{\left( II_{\ast
}I^{\prime }I_{\ast }^{\prime }\right) ^{\delta /4-1/2}}W(\boldsymbol{\xi },%
\boldsymbol{\xi }^{\prime },I,I^{\prime }\left\vert \boldsymbol{\xi }_{\ast
},\boldsymbol{\xi }_{\ast }^{\prime },I_{\ast },I_{\ast }^{\prime }\right.
)\,d\mathbf{g}^{\prime }d\mathbf{h\,}dI^{\prime }dI_{\ast }^{\prime } \\
&=&\!\!\left( \frac{I}{I_{\ast }}\right) ^{\delta
/4-1/2}\!\!\!\!\int_{\left( \mathbb{R}^{3}\right) ^{\perp _{\mathbf{n}%
}}\times \mathbb{R}_{+}^{2}}\!\!2\left( M^{\prime }M_{\ast }^{\prime
}\right) ^{1/2}\!\left( \frac{I^{\prime }}{I_{\ast }^{\prime }}\right)
^{\delta /4-1/2}\!\!\widetilde{\sigma }\frac{\left\vert \widetilde{\mathbf{g}%
}\right\vert \mathbf{1}_{\left\vert \widetilde{\mathbf{g}}\right\vert
^{2}>4\Delta I_{\ast }}}{\left\vert \mathbf{g}_{\ast }\right\vert \left\vert 
\mathbf{g}\right\vert }\,d\mathbf{w}dI^{\prime }dI_{\ast }^{\prime }\text{.}
\end{eqnarray*}%
Here, see Figure $\ref{fig2}$,%
\begin{equation*}
\left\{ 
\begin{array}{c}
\boldsymbol{\xi }^{\prime }=\boldsymbol{\xi }_{\ast }+\mathbf{w}-\chi 
\mathbf{n} \\ 
\boldsymbol{\xi }_{\ast }^{\prime }=\boldsymbol{\xi }+\mathbf{w}-\chi 
\mathbf{n}%
\end{array}%
\right. \text{, with }\chi =\frac{\Delta I_{\ast }}{m\left\vert \mathbf{g}%
\right\vert }\text{,}
\end{equation*}%
implying that%
\begin{eqnarray*}
\frac{\left\vert \boldsymbol{\xi }^{\prime }\right\vert ^{2}}{2}+\frac{%
\left\vert \boldsymbol{\xi }_{\ast }^{\prime }\right\vert ^{2}}{2}
&=&\left\vert \frac{\boldsymbol{\xi +\xi }_{\ast }}{2}-\chi \mathbf{n}+%
\mathbf{w}\right\vert ^{2}+\frac{\left\vert \boldsymbol{\xi }-\boldsymbol{%
\xi }_{\ast }\right\vert ^{2}}{4} \\
&=&\left\vert \frac{\left( \boldsymbol{\xi +\xi }_{\ast }\right) _{\perp _{%
\boldsymbol{n}}}}{2}+\mathbf{w}\right\vert ^{2}+\left( \frac{\left( 
\boldsymbol{\xi +\xi }_{\ast }\right) _{\mathbf{n}}}{2}-\chi \right) ^{2}+%
\frac{\left\vert \mathbf{g}\right\vert ^{2}}{4} \\
&=&\left\vert \frac{\left( \boldsymbol{\xi +\xi }_{\ast }\right) _{\perp _{%
\boldsymbol{n}}}}{2}+\mathbf{w}\right\vert ^{2}+\frac{\left( \left\vert 
\boldsymbol{\xi }_{\ast }\right\vert ^{2}-\left\vert \boldsymbol{\xi }%
\right\vert ^{2}+2\chi \left\vert \boldsymbol{\xi }-\boldsymbol{\xi }_{\ast
}\right\vert \right) ^{2}}{4\left\vert \boldsymbol{\xi }-\boldsymbol{\xi }%
_{\ast }\right\vert ^{2}}+\frac{\left\vert \mathbf{g}\right\vert ^{2}}{4}%
\text{,\ }
\end{eqnarray*}%
with%
\begin{eqnarray*}
\left( \boldsymbol{\xi +\xi }_{\ast }\right) _{\mathbf{n}} &=&\left( 
\boldsymbol{\xi +\xi }_{\ast }\right) \cdot \mathbf{n}=\frac{\left\vert 
\boldsymbol{\xi }\right\vert ^{2}-\left\vert \boldsymbol{\xi }_{\ast
}\right\vert ^{2}}{\left\vert \boldsymbol{\xi }-\boldsymbol{\xi }_{\ast
}\right\vert }\text{ and} \\
\text{\ }\left( \boldsymbol{\xi +\xi }_{\ast }\right) _{\perp _{\boldsymbol{n%
}}} &=&\boldsymbol{\xi +\xi }_{\ast }-\left( \boldsymbol{\xi +\xi }_{\ast
}\right) _{\mathbf{n}}\mathbf{n}\text{.}
\end{eqnarray*}%
Note that for any number $s$, such that $s\geq -1/2$,%
\begin{equation}
\frac{\left( II_{\ast }\right) ^{\delta /4-1/2}}{\widetilde{E}^{\delta
/2+s-1/2}}\leq \frac{I^{\delta /4-1/2-\kappa /2}}{I_{\ast }^{\delta
/4+s-\kappa /2}}\text{ for }0\leq \kappa \leq \delta +2s-1\text{,}
\label{b5}
\end{equation}%
where $\widetilde{E}=m\left\vert \widetilde{\mathbf{g}}\right\vert
^{2}/4+I+I^{\prime }=m\left\vert \mathbf{g}_{\ast }\right\vert
^{2}/4+I_{\ast }+I_{\ast }^{\prime }$.

By bound $\left( \ref{b5}\right) $ and assumption $\left( \ref{est1}\right) $%
, for any numbers $s$ and $\kappa $, such that $\ -1/2\leq s\leq \delta /2$
and $0\leq \kappa \leq \delta +2s-1$, 
\begin{eqnarray}
&&k_{2}^{2}(\boldsymbol{\xi },\boldsymbol{\xi }_{\ast },I,I_{\ast })  \notag
\\
&\leq &\frac{C}{\left\vert \mathbf{g}\right\vert ^{2}}\frac{I^{\delta
/2-1-\kappa }}{I_{\ast }^{\delta /2+2s-\kappa }}\left( \int_{\mathbb{R}%
_{+}^{2}}\exp \left( -m\frac{\left( \left\vert \boldsymbol{\xi }_{\ast
}\right\vert ^{2}-\left\vert \boldsymbol{\xi }\right\vert ^{2}+2\chi
\left\vert \mathbf{g}\right\vert \right) ^{2}}{8\left\vert \mathbf{g}%
\right\vert ^{2}}-\frac{m}{8}\left\vert \mathbf{g}\right\vert ^{2}\right)
\right.   \notag \\
&&\times \int_{\left( \mathbb{R}^{3}\right) ^{\perp _{\mathbf{n}}}}\left( 1+%
\frac{1}{\widetilde{\Psi }^{1-\gamma /2}}\right) \exp \left( -\frac{m}{2}%
\left\vert \frac{\left( \boldsymbol{\xi +\xi }_{\ast }\right) _{\perp _{%
\boldsymbol{n}}}}{2}+\mathbf{w}\right\vert ^{2}\right) d\mathbf{w}  \notag \\
&&\left. \,\times \mathbf{\,}e^{-\left( I^{\prime }+I_{\ast }^{\prime
}\right) /2}\frac{\left( I^{\prime }I_{\ast }^{\prime }\right) ^{\delta
/4-1/2}}{\widetilde{E}^{\delta /2-s}}dI^{\prime }dI_{\ast }^{\prime }\right)
^{2}  \notag \\
&\leq &\frac{C}{\left\vert \mathbf{g}\right\vert ^{2}}\frac{I^{\delta
/2-1-\kappa }}{I_{\ast }^{\delta /2+2s-\kappa }}\left( \int_{\mathbb{R}%
_{+}^{2}}\exp \left( -\frac{m}{8}\left( \left\vert \mathbf{g}\right\vert
+2\left\vert \boldsymbol{\xi }\right\vert \cos \varphi +2\chi \right) ^{2}-%
\frac{m}{8}\left\vert \mathbf{g}\right\vert ^{2}\right) \right.   \notag \\
&&\times \left. \frac{e^{-\left( I^{\prime }+I_{\ast }^{\prime }\right) /2}}{%
\left( I^{\prime }I_{\ast }^{\prime }\right) ^{1/2-s/2}}dI^{\prime }dI_{\ast
}^{\prime }\right) ^{2}\text{, where }\widetilde{\Psi }=\left\vert 
\widetilde{\mathbf{g}}\right\vert \left\vert \mathbf{g}_{\ast }\right\vert 
\text{ and }\cos \varphi =\mathbf{n}\cdot \frac{\boldsymbol{\xi }}{%
\left\vert \boldsymbol{\xi }\right\vert }\text{,}  \label{b2a}
\end{eqnarray}%
since 
\begin{eqnarray*}
&&\int_{\left( \mathbb{R}^{3}\right) ^{\perp _{\mathbf{n}}}}\left( 1+\frac{1%
}{\widetilde{\Psi }^{1-\gamma /2}}\right) \exp \left( -\frac{m}{2}\left\vert 
\frac{\left( \boldsymbol{\xi +\xi }_{\ast }\right) _{\perp _{\boldsymbol{n}}}%
}{2}+\mathbf{w}\right\vert ^{2}\right) d\mathbf{w} \\
&\leq &C\int_{\left( \mathbb{R}^{3}\right) ^{\perp _{\mathbf{n}}}}\left(
1+\left\vert \mathbf{w}\right\vert ^{\gamma -2}\right) \exp \left( -\frac{m}{%
2}\left\vert \frac{\left( \boldsymbol{\xi +\xi }_{\ast }\right) _{\perp _{%
\boldsymbol{n}}}}{2}+\mathbf{w}\right\vert ^{2}\right) \,d\mathbf{w\,} \\
&\leq &C\left( \int_{\left\vert \mathbf{w}\right\vert \leq 1}1+\left\vert 
\mathbf{w}\right\vert ^{\gamma -2}d\mathbf{w\,}+2\int_{\left\vert \mathbf{w}%
\right\vert \geq 1}\exp \left( -\frac{m}{2}\left\vert \frac{\left( 
\boldsymbol{\xi +\xi }_{\ast }\right) _{\perp _{\boldsymbol{n}}}}{2}+\mathbf{%
w}\right\vert ^{2}\right) d\mathbf{w}\right)  \\
&\leq &C\left( \int_{\left\vert \mathbf{w}\right\vert \leq 1}1+\left\vert 
\mathbf{w}\right\vert ^{\gamma -2}\,d\mathbf{w}\,+\int_{\left( \mathbb{R}%
^{3}\right) ^{\perp _{\mathbf{n}}}}e^{-\left\vert \widetilde{\mathbf{w}}%
\right\vert ^{2}}\,d\widetilde{\mathbf{w}}\right)  \\
&=&C\left( \int_{0}^{1}r+r^{\gamma -1}\,dr+\int_{0}^{\infty
}re^{-r^{2}}\,dr\right) =C\text{.}
\end{eqnarray*}%
For any numbers $s$ and $\kappa $, such that $\ -1/2\leq s\leq \delta /2$
and $0\leq \kappa \leq \delta +2s-1$, by the bound $\left( \ref{b2a}\right) $
on $k_{2}^{2}$ and the Cauchy-Schwarz inequality, 
\begin{eqnarray}
&&k_{2}^{2}(\boldsymbol{\xi },\boldsymbol{\xi }_{\ast },I,I_{\ast })  \notag
\\
&\leq &\frac{C}{\left\vert \mathbf{g}\right\vert ^{2}}\frac{I^{\delta
/2-1-\kappa }}{I_{\ast }^{\delta /2+2s-\kappa }}\int_{\mathbb{R}_{+}^{2}}%
\mathbf{\,}\frac{e^{-\left( I^{\prime }+I_{\ast }^{\prime }\right) /2}}{%
\left( I^{\prime }I_{\ast }^{\prime }\right) ^{1/2-s/2}}dI^{\prime }dI_{\ast
}^{\prime }  \notag \\
&&\times \int_{\mathbb{R}_{+}^{2}}\exp \left( -\frac{m}{4}\left( \left\vert 
\mathbf{g}\right\vert +2\left\vert \boldsymbol{\xi }\right\vert \cos \varphi
+2\chi \right) ^{2}-\frac{m}{4}\left\vert \mathbf{g}\right\vert ^{2}\right) 
\frac{e^{-\left( I^{\prime }+I_{\ast }^{\prime }\right) /2}}{\left(
I^{\prime }I_{\ast }^{\prime }\right) ^{1/2-s/2}}dI^{\prime }dI_{\ast
}^{\prime }  \notag \\
&=&\frac{C}{\left\vert \mathbf{g}\right\vert ^{2}}\frac{I^{\delta
/2-1-\kappa }}{I_{\ast }^{\delta /2+2s-\kappa }}\int_{\mathbb{R}%
_{+}^{2}}\exp \left( -\frac{m}{4}\left( \left\vert \mathbf{g}\right\vert
+2\left\vert \boldsymbol{\xi }\right\vert \cos \varphi +2\chi \right) ^{2}-%
\frac{m}{4}\left\vert \mathbf{g}\right\vert ^{2}\right)   \notag \\
&&\times \frac{e^{-\left( I^{\prime }+I_{\ast }^{\prime }\right) /2}}{\left(
I^{\prime }I_{\ast }^{\prime }\right) ^{1/2-s/2}}dI^{\prime }dI_{\ast
}^{\prime }\text{, with }\cos \varphi =\mathbf{n}\cdot \dfrac{\boldsymbol{%
\xi }}{\left\vert \boldsymbol{\xi }\right\vert },  \label{b2}
\end{eqnarray}%
since, 
\begin{equation}
\int_{\mathbb{R}_{+}^{2}}\mathbf{\,}\frac{e^{-\left( I^{\prime }+I_{\ast
}^{\prime }\right) /2}}{\left( I^{\prime }I_{\ast }^{\prime }\right)
^{1/2-s/2}}dI^{\prime }dI_{\ast }^{\prime }=\left( \int_{0}^{\infty }\mathbf{%
\,}\frac{e^{-I/2}}{I^{1/2-s/2}}dI\right) ^{2}=C\text{,}  \label{i1}
\end{equation}%
implying also%
\begin{equation}
k_{2}^{2}(\boldsymbol{\xi },\boldsymbol{\xi }_{\ast },I,I_{\ast })\leq \frac{%
C}{\left\vert \mathbf{g}\right\vert ^{2}}\frac{I^{\delta /2-1-\kappa }}{%
I_{\ast }^{\delta /2+2s-\kappa }}\exp \left( -\frac{m}{4}\left\vert \mathbf{g%
}\right\vert ^{2}\right) \text{.}  \label{b2b}
\end{equation}

Note that by letting $\kappa =(\delta -1)/2+s$ 
\begin{equation}
\frac{I^{\delta /2-1-\kappa }}{I_{\ast }^{\delta /2+2s-\kappa }}=\left(
II_{\ast }\right) ^{-1/2-s}=\left\{ 
\begin{array}{l}
\left( II_{\ast }\right) ^{-5/4}\text{ for }s=3/4 \\ 
1\text{ for }s=-1/2%
\end{array}%
\right. \text{,}  \label{b3}
\end{equation}%
while by letting $s=1/8$ 
\begin{equation}
\frac{I^{\delta /2-1-\kappa }}{I_{\ast }^{\delta /2+2s-\kappa }}=\frac{%
I^{\delta /2-1-\kappa }}{I_{\ast }^{\delta /2+1/4-\kappa }}=\left\{ 
\begin{array}{l}
I_{\ast }^{-5/4}\text{ for }\kappa =\delta /2-1 \\ 
I^{-5/4}\text{ for }\kappa =\delta /2+1/4%
\end{array}%
\right. \text{.}  \label{b4}
\end{equation}%
Moreover, by letting $s=3/4$%
\begin{equation}
\frac{I^{\delta /2-1-\kappa }}{I_{\ast }^{\delta /2+2s-\kappa }}=\frac{%
I^{\delta /2-1-\kappa }}{I_{\ast }^{\delta /2+3/2-\kappa }}=\left\{ 
\begin{array}{l}
I_{\ast }^{-5/2}\text{ for }\kappa =\delta /2-1 \\ 
I^{-5/2}\text{ for }\kappa =\delta /2+3/2%
\end{array}%
\right. \text{.}  \label{b4b}
\end{equation}

Hence, by the bound $\left( \ref{b2b}\right) $ on $k_{2}^{2}$, together with
expressions $\left( \ref{b3}\right) $ and $\left( \ref{b4}\right) $, 
\begin{eqnarray}
&&k_{2}^{2}(\boldsymbol{\xi },\boldsymbol{\xi }_{\ast },I,I_{\ast })  \notag
\\
&\leq &\frac{C}{\left\vert \mathbf{g}\right\vert ^{2}}\exp \left( -\frac{m}{4%
}\left\vert \mathbf{g}\right\vert ^{2}\right) \left( \mathbf{1}_{I\leq
1}+I^{-5/4}\mathbf{1}_{I\geq 1}\right) \left( \mathbf{1}_{I_{\ast }\leq
1}+I_{\ast }^{-5/4}\mathbf{1}_{I_{\ast }\geq 1}\right) \text{.}  \label{b3a}
\end{eqnarray}

To show that $k_{2}(\boldsymbol{\xi },\boldsymbol{\xi }_{\ast },I,I_{\ast })%
\mathbf{1}_{\mathfrak{h}_{N}}\in L^{2}\left( d\boldsymbol{\xi \,}\,d%
\boldsymbol{\xi }_{\ast }dI\,dI_{\ast }\right) $ for any (non-zero) natural
number $N$, separate the integration domain\ of the integral of $k_{2}^{2}(%
\boldsymbol{\xi },\boldsymbol{\xi }_{\ast },I,I_{\ast })$ over $\left( 
\mathbb{R}^{3}\times \mathbb{R}_{+}\right) ^{2}$ in two separate domains 
\begin{equation*}
\left\{ \left( \mathbb{R}^{3}\times \mathbb{R}_{+}\right) ^{2}\text{; }%
\left\vert \mathbf{g}\right\vert \geq \left\vert \boldsymbol{\xi }%
\right\vert \right\} \text{ and }\left\{ \left( \mathbb{R}^{3}\times \mathbb{%
R}_{+}\right) ^{2}\text{; }\left\vert \mathbf{g}\right\vert \leq \left\vert 
\boldsymbol{\xi }\right\vert \right\} .
\end{equation*}%
The integral of $k_{2}^{2}$ over the domain $\left\{ \left( \mathbb{R}%
^{3}\times \mathbb{R}_{+}\right) ^{2}\text{; }\left\vert \mathbf{g}%
\right\vert \geq \left\vert \boldsymbol{\xi }\right\vert \right\} $ will be
bounded, since, by the bound $\left( \ref{b3a}\right) $,%
\begin{eqnarray*}
&&\int_{0}^{\infty }\int_{0}^{\infty }\int_{\left\vert \mathbf{g}\right\vert
\geq \left\vert \boldsymbol{\xi }\right\vert }k_{2}^{2}(\boldsymbol{\xi },%
\boldsymbol{\xi }_{\ast },I,I_{\ast })\,d\boldsymbol{\xi }d\boldsymbol{\xi }%
_{\ast }dI\mathbf{\,}dI_{\ast } \\
&\leq &C\int_{\left\vert \mathbf{g}\right\vert \geq \left\vert \boldsymbol{%
\xi }\right\vert }\frac{e^{-m\left\vert \mathbf{g}\right\vert ^{2}/4}}{%
\left\vert \mathbf{g}\right\vert ^{2}}\mathbf{\,}d\mathbf{g\,}d\boldsymbol{%
\xi \,}\left( \int_{0}^{1}dI+\int_{1}^{\infty }I^{-5/4}dI\right) ^{2} \\
&\leq &C\int_{\left\vert \mathbf{g}\right\vert \geq \left\vert \boldsymbol{%
\xi }\right\vert }\frac{e^{-m\left\vert \mathbf{g}\right\vert ^{2}/4}}{%
\left\vert \mathbf{g}\right\vert ^{2}}\mathbf{\,}d\mathbf{g\,}d\boldsymbol{%
\xi \,}=C\int_{0}^{\infty }\int_{r}^{\infty }e^{-mR^{2}/4}r^{2}dR\mathbf{\,}%
dr\boldsymbol{\,} \\
&\leq &C\int_{0}^{\infty }e^{-mR^{2}/8}dR\int_{0}^{\infty
}e^{-mr^{2}/8}r^{2}dr=C
\end{eqnarray*}%
Regarding the second domain, consider the truncated domains 
\begin{equation*}
\left\{ \left( \mathbb{R}^{3}\times \mathbb{R}_{+}\right) ^{2}\text{; }%
\left\vert \mathbf{g}\right\vert \leq \left\vert \boldsymbol{\xi }%
\right\vert \leq N\right\} 
\end{equation*}%
for (non-zero) natural numbers $N$. Then by the bound $\left( \ref{b2}%
\right) $ on $k_{2}^{2}$, together with expressions $\left( \ref{i1}\right) $%
,$\left( \ref{b3}\right) $, $\left( \ref{b4}\right) $, 
\begin{eqnarray*}
\int_{\mathbb{R}_{+}^{2}}\int_{\left\vert \mathbf{g}\right\vert \leq
\left\vert \boldsymbol{\xi }\right\vert \leq N}k_{2}^{2}(\boldsymbol{\xi },%
\boldsymbol{\xi }_{\ast },I,I_{\ast })\,d\boldsymbol{\xi }d\boldsymbol{\xi }%
_{\ast }dIdI_{\ast }&\leq& CN^{2}\left( \int_{0}^{1}dI+\int_{1}^{\infty
}I^{-5/4}dI\right) ^{2} \\
&=&CN^{2}\text{,}
\end{eqnarray*}%
since%
\begin{eqnarray*}
&&\int_{\left\vert \mathbf{g}\right\vert \leq \left\vert \boldsymbol{\xi }%
\right\vert \leq N}\frac{C}{\left\vert \mathbf{g}\right\vert ^{2}}\exp
\left( -\frac{m}{4}\left( \left\vert \mathbf{g}\right\vert +2\left\vert 
\boldsymbol{\xi }\right\vert \cos \varphi +2\chi \right) ^{2}-\frac{m}{4}%
\left\vert \mathbf{g}\right\vert ^{2}\right) d\boldsymbol{\xi }\mathbf{\,}d%
\mathbf{g} \\
&=&C\int_{0}^{N}\int_{0}^{r}\int_{0}^{\pi }r^{2}\exp \left( -\frac{m}{4}%
\left( R+2r\cos \varphi +2\chi _{R}\right) ^{2}-\frac{m}{4}R^{2}\right) \sin
\varphi \,d\varphi \mathbf{\,}dR\mathbf{\,}dr \\
&=&C\int_{0}^{N}\int_{0}^{r}\int_{R+2\chi _{R}-2r}^{R+2\chi
_{R}+2r}re^{-m\eta ^{2}/4}e^{-mR^{2}/4}d\eta \mathbf{\,}dR\mathbf{\,}dr \\
&\leq &C\int_{0}^{N}r\,dr\int_{0}^{\infty }e^{-mR^{2}/4}dR\int_{-\infty
}^{\infty }e^{-m\eta ^{2}/4}d\eta =CN^{2}\text{, with }\chi _{R}=\frac{%
\Delta I_{\ast }}{mR}\text{.}
\end{eqnarray*}

Furthermore, the integral of $k_{2}(\boldsymbol{\xi },\boldsymbol{\xi }%
_{\ast },I,I_{\ast })$ with respect to $\left( \boldsymbol{\xi },I\right) $
over $\mathbb{R}^{3}\times \mathbb{R}_{+}$ is bounded in $\left( \boldsymbol{%
\xi }_{\ast },I_{\ast }\right) $. Indeed, by the bound $\left( \ref{b2b}%
\right) $ on $k_{2}^{2}$, together with expressions on $k_{2}^{2}$, together
with expressions $\left( \ref{b3}\right) $ and $\left( \ref{b4b}\right) $,%
\begin{equation*}
k_{2}(\boldsymbol{\xi },\boldsymbol{\xi }_{\ast },I,I_{\ast })\leq \frac{C}{%
\left\vert \mathbf{g}\right\vert }\exp \left( -\frac{m}{8}\left\vert \mathbf{%
g}\right\vert ^{2}\right) \left( \mathbf{1}_{I_{\ast }\leq 1}+I_{\ast
}^{-5/4}\mathbf{1}_{I_{\ast }\geq 1}\right) \text{,}
\end{equation*}%
why the following bound on the integral of $k_{2}$ with respect to $\left( 
\boldsymbol{\xi }_{\ast },I_{\ast }\right) $ over the domain $\left\{ 
\mathbb{R}^{3}\!\!\times \mathbb{R}_{+}\text{;}\left\vert \mathbf{g}%
\right\vert \!\geq \!\left\vert \boldsymbol{\xi }\right\vert \right\} $ can
be obtained for $\left\vert \boldsymbol{\xi }\right\vert \neq 0$%
\begin{eqnarray*}
&&\int_{0}^{\infty }\int_{\left\vert \mathbf{g}\right\vert \geq \left\vert 
\boldsymbol{\xi }\right\vert }k_{2}(\boldsymbol{\xi },\boldsymbol{\xi }%
_{\ast },I,I_{\ast })\,d\boldsymbol{\xi }_{\ast }dI_{\ast } \\
&\leq &\frac{C}{\left\vert \boldsymbol{\xi }\right\vert }\int_{\left\vert 
\mathbf{g}\right\vert \geq \left\vert \boldsymbol{\xi }\right\vert
}e^{-m\left\vert \mathbf{g}\right\vert ^{2}/8}d\mathbf{g}\left(
\int_{0}^{1}dI_{\ast }+\int_{1}^{\infty }\frac{dI_{\ast }}{I_{\ast }^{5/4}}%
\right)  \\
&\leq &\frac{C}{\left\vert \boldsymbol{\xi }\right\vert }\int_{0}^{\infty
}e^{-mr^{2}/8}r^{2}dr\int_{\mathbb{S}^{2}}\,d\boldsymbol{\omega }=\frac{C}{%
\left\vert \boldsymbol{\xi }\right\vert }\text{.}
\end{eqnarray*}%
Moreover, over the domain $\left\{ \mathbb{R}^{3}\times \mathbb{R}_{+}\text{%
; }\left\vert \mathbf{g}\right\vert \leq \left\vert \boldsymbol{\xi }%
\right\vert \right\} $, by the bound $\left( \ref{b2a}\right) $ on $k_{2}^{2}
$, and expressions $\left( \ref{i1}\right) $, $\left( \ref{b3}\right) $, $%
\left( \ref{b4b}\right) $, 
\begin{equation*}
\int_{0}^{\infty }\int_{\left\vert \mathbf{g}\right\vert \leq \left\vert 
\boldsymbol{\xi }\right\vert }k_{2}(\boldsymbol{\xi },\boldsymbol{\xi }%
_{\ast },I,I_{\ast })\,d\boldsymbol{\xi }_{\ast }dI_{\ast }\leq \frac{C}{%
\left\vert \boldsymbol{\xi }\right\vert }\left( \int_{0}^{1}dI_{\ast
}+\int_{1}^{\infty }I_{\ast }^{-5/4}dI_{\ast }\right) =\frac{C}{\left\vert 
\boldsymbol{\xi }\right\vert }\text{,}
\end{equation*}%
since%
\begin{eqnarray*}
&&\int_{\left\vert \mathbf{g}\right\vert \leq \left\vert \boldsymbol{\xi }%
\right\vert }\frac{C}{\left\vert \mathbf{g}\right\vert }\exp \left( -\frac{m%
}{8}\left( \left\vert \mathbf{g}\right\vert +2\left\vert \boldsymbol{\xi }%
\right\vert \cos \varphi +2\chi \right) ^{2}-\frac{m}{8}\left\vert \mathbf{g}%
\right\vert ^{2}\right) \,d\mathbf{g} \\
&=&C\int_{0}^{\left\vert \boldsymbol{\xi }\right\vert }\int_{0}^{\pi }R\exp
\left( -\frac{m}{8}\left( R+2\left\vert \boldsymbol{\xi }\right\vert \cos
\varphi +2\chi _{R}\right) ^{2}-\frac{m}{8}R^{2}\right) \sin \varphi
\,d\varphi \mathbf{\,}dR \\
&=&\frac{C}{\left\vert \boldsymbol{\xi }\right\vert }\int_{0}^{\left\vert 
\boldsymbol{\xi }\right\vert }\int_{R+2\chi _{R}-2\left\vert \boldsymbol{\xi 
}\right\vert }^{R+2\chi _{R}+2\left\vert \boldsymbol{\xi }\right\vert
}Re^{-m\eta ^{2}/8}e^{-mR^{2}/8}d\eta dR \\
&\leq &\frac{C}{\left\vert \boldsymbol{\xi }\right\vert }\int_{0}^{\infty
}Re^{-mR^{2}/8}\,dR\int_{-\infty }^{\infty }e^{-m\eta ^{2}/8}d\eta =\frac{C}{%
\left\vert \boldsymbol{\xi }\right\vert }\text{, with }\chi _{R}=\frac{%
\Delta I_{\ast }}{mR}\text{.}
\end{eqnarray*}%
However, due to the symmetry $k_{2}(\boldsymbol{\xi },\boldsymbol{\xi }%
_{\ast },I,I_{\ast })=k_{2}(\boldsymbol{\xi }_{\ast },\boldsymbol{\xi }%
,I_{\ast },I)$ $\left( \ref{sa2}\right) $, also 
\begin{equation*}
\int_{0}^{\infty }\int_{\mathbb{R}^{3}}k_{2}(\boldsymbol{\xi },\boldsymbol{%
\xi }_{\ast },I,I_{\ast })\,d\boldsymbol{\xi \,}dI\leq \frac{C}{\left\vert 
\boldsymbol{\xi }_{\ast }\right\vert }\text{.}
\end{equation*}%
Therefore, if $\left\vert \boldsymbol{\xi }_{\ast }\right\vert \geq 1$, then%
\begin{equation*}
\int_{0}^{\infty }\int_{\mathbb{R}^{3}}k_{2}(\boldsymbol{\xi },\boldsymbol{%
\xi }_{\ast },I,I_{\ast })\,d\boldsymbol{\xi \,}dI\,\leq \frac{C}{\left\vert 
\boldsymbol{\xi }_{\ast }\right\vert }\leq C\text{.}
\end{equation*}%
Otherwise, if $\left\vert \boldsymbol{\xi }_{\ast }\right\vert \leq 1$,
then, by the bounds bound $\left( \ref{b2a}\right) $ on $k_{2}^{2}$, and
expressions $\left( \ref{i1}\right) $, $\left( \ref{b3}\right) $, $\left( %
\ref{b4b}\right) $,%
\begin{eqnarray*}
&&\int_{0}^{\infty }\int_{\mathbb{R}^{3}}k_{2}(\boldsymbol{\xi },\boldsymbol{%
\xi }_{\ast },I,I_{\ast })\,d\boldsymbol{\xi \,}dI=\int_{0}^{\infty }\int_{%
\mathbb{R}^{3}}k_{2}(\boldsymbol{\xi }_{\ast },\boldsymbol{\xi },I_{\ast
},I)\,d\boldsymbol{\xi \,}dI\,\, \\
&\leq &C\left( \int_{0}^{1}dI+\int_{1}^{\infty }I^{-5/4}dI\right) \leq C
\end{eqnarray*}%
since%
\begin{eqnarray*}
&&\int_{\mathbb{R}^{3}}\frac{C}{\left\vert \mathbf{g}\right\vert }\exp
\left( -\frac{m}{8}\left( \left\vert \mathbf{g}\right\vert +2\left\vert 
\boldsymbol{\xi }_{\ast }\right\vert \cos \varphi _{\ast }-2\chi \right)
^{2}-\frac{m}{8}\left\vert \mathbf{g}\right\vert ^{2}\right) \,d\mathbf{g} \\
&=&C\int_{0}^{\infty }\int_{0}^{\pi }R\exp \left( -\frac{m}{8}\left(
R+2\left\vert \boldsymbol{\xi }\right\vert \cos \varphi _{\ast }-2\chi
_{R}\right) ^{2}-\frac{m}{8}R^{2}\right) \sin \varphi _{\ast }\,d\varphi
_{\ast }\boldsymbol{\,}dR \\
&\leq &\frac{C}{\left\vert \boldsymbol{\xi }_{\ast }\right\vert }%
\int_{0}^{\infty }Re^{-mR^{2}/8}dR\int_{R-2\chi _{R}-2\left\vert \boldsymbol{%
\xi }_{\ast }\right\vert }^{R-2\chi _{R}+2\left\vert \boldsymbol{\xi }_{\ast
}\right\vert }\,dt=C\text{, with }\chi _{R}=\frac{\Delta I_{\ast }}{mR}\text{
and }\\
&&\cos \varphi _{\ast }=\mathbf{n}\cdot \frac{\boldsymbol{\xi }_{\ast }}{%
\left\vert \boldsymbol{\xi }_{\ast }\right\vert }\text{.}
\end{eqnarray*}

Furthermore,%
\begin{eqnarray*}
&&\sup_{\left( \boldsymbol{\xi },I\right) \in \mathbb{R}^{3}\times \mathbb{R}%
_{+}}\int_{\mathbb{R}^{3}\times \mathbb{R}_{+}}k_{2}(\boldsymbol{\xi },%
\boldsymbol{\xi }_{\ast },I,I_{\ast })-k_{2}(\boldsymbol{\xi },\boldsymbol{%
\xi }_{\ast },I,I_{\ast })\mathbf{1}_{\mathfrak{h}_{N}}\,d\boldsymbol{\xi }%
_{\ast }dI_{\ast } \\
&\leq &\sup_{\left( \boldsymbol{\xi },I\right) \in \mathbb{R}^{3}\times 
\mathbb{R}_{+}}\int_{0}^{\infty }\int_{\left\vert \mathbf{g}\right\vert \leq 
\frac{1}{N}}k_{2}(\boldsymbol{\xi },\boldsymbol{\xi }_{\ast },I,I_{\ast })\,d%
\boldsymbol{\xi }_{\ast }dI_{\ast } \\
&&+\sup_{\left\vert \boldsymbol{\xi }\right\vert \geq N}\int_{0}^{\infty
}\int_{\mathbb{R}^{3}}k_{2}(\boldsymbol{\xi },\boldsymbol{\xi }_{\ast
},I,I_{\ast })\,d\boldsymbol{\xi }_{\ast }dI_{\ast } \\
&\leq &\int_{\left\vert \mathbf{g}\right\vert \leq \frac{1}{N}}\frac{C}{%
\left\vert \mathbf{g}\right\vert }\,d\mathbf{g}\left( \int_{0}^{1}dI_{\ast
}+\int_{1}^{\infty }\frac{dI_{\ast }}{I_{\ast }^{5/4}}\right) +\frac{C}{N}%
\leq C\left( \int_{0}^{\frac{1}{N}}R\,dR+\frac{1}{N}\right) \\
&=&C\left( \frac{1}{N^{2}}+\frac{1}{N}\right) \rightarrow 0\text{ as }%
N\rightarrow \infty \text{.}
\end{eqnarray*}

Hence, by Lemma \ref{LGD}, the operator 
\begin{equation*}
K_{2}=\int_{\mathbb{R}^{3}\times \mathbb{R}_{+}}k_{2}(\boldsymbol{\xi },%
\boldsymbol{\xi }_{\ast },I,I_{\ast })\,h_{\ast }\,d\boldsymbol{\xi }_{\ast
}dI_{\ast }
\end{equation*}%
is compact on $L^{2}\left( d\boldsymbol{\xi \,}dI\right) $.

Concluding, the operator $K=K_{2}-K_{1}$ is a compact self-adjoint operator
on $L^{2}\left( d\boldsymbol{\xi \,\,}dI\right) $. The self-adjointness is
due to the symmetry relations $\left( \ref{sa1}\right) ,\left( \ref{sa2}%
\right) $, cf. \cite[p.198]{Yoshida-65}.
\end{proof}

\section{\label{PT2}Bounds on the collision frequency}

This section concerns the proof of Theorem \ref{Thm2}. Note that throughout
the proof, $C$ will denote a generic positive constant.

\begin{proof}
Under assumption $\left( \ref{e1}\right) $ the collision frequency $\upsilon 
$ equals%
\begin{eqnarray*}
\upsilon &=&\int_{\left( \mathbb{R}^{3}\times \mathbb{R}_{+}\right) ^{3}}%
\frac{M_{\ast }}{\left( II_{\ast }\right) ^{\delta /2-1}}W(\boldsymbol{\xi },%
\boldsymbol{\xi }_{\ast },I,I_{\ast }\left\vert \boldsymbol{\xi }^{\prime },%
\boldsymbol{\xi }_{\ast }^{\prime },I^{\prime },I_{\ast }^{\prime }\right.
)\,d\boldsymbol{\xi }_{\ast }d\boldsymbol{\xi }^{\prime }d\boldsymbol{\xi }%
_{\ast }^{\prime }dI_{\ast }dI^{\prime }dI_{\ast }^{\prime } \\
&=&C\int_{\left( \mathbb{R}^{3}\times \mathbb{R}_{+}\right) ^{3}}M_{\ast
}\sigma \frac{\left\vert \mathbf{g}\right\vert }{\left\vert \mathbf{g}%
^{\prime }\right\vert ^{2}}\delta _{3}\left( \mathbf{G}-\mathbf{G}^{\prime
}\right) \widehat{\delta }_{1}\,d\boldsymbol{\xi }_{\ast }d\mathbf{G}%
^{\prime }d\mathbf{g}^{\prime }dI_{\ast }dI^{\prime }dI_{\ast }^{\prime } \\
&=&C\int_{\mathbb{R}^{3}\times \left( \mathbb{R}_{+}\right) ^{4}}I_{\ast
}^{\delta /2-1}e^{-I_{\ast }}e^{-m\left\vert \boldsymbol{\xi }_{\ast
}\right\vert ^{2}/2}\sigma \left\vert \mathbf{g}\right\vert \widehat{\delta }%
_{1}\,d\boldsymbol{\xi }_{\ast }d\left\vert \mathbf{g}^{\prime }\right\vert
dI_{\ast }dI^{\prime }dI_{\ast }^{\prime }\int_{\mathbb{S}^{2}}\,d%
\boldsymbol{\omega } \\
&=&C\int_{\mathbb{R}^{3}\times \left( \mathbb{R}_{+}\right) ^{4}}e^{-I_{\ast
}}e^{-m\left\vert \boldsymbol{\xi }_{\ast }\right\vert ^{2}/2}\dfrac{%
\left\vert \mathbf{g}^{\prime }\right\vert \left( I_{\ast }I^{\prime
}I_{\ast }^{\prime }\right) ^{\delta /2-1}}{E^{\delta -1/2}}\widehat{\delta }%
_{1}\,d\boldsymbol{\xi }_{\ast }d\left\vert \mathbf{g}^{\prime }\right\vert
dI_{\ast }dI^{\prime }dI_{\ast }^{\prime } \\
&=&C\int_{\mathbb{R}^{3}\times \left( \mathbb{R}_{+}\right) ^{3}}e^{-I_{\ast
}}e^{-m\left\vert \boldsymbol{\xi }_{\ast }\right\vert ^{2}/2}\dfrac{\sqrt{%
E-\left( I^{\prime }+I_{\ast }^{\prime }\right) }}{E^{\delta -1/2}}\left(
I_{\ast }I^{\prime }I_{\ast }^{\prime }\right) ^{\delta /2-1} \\
&&\times \mathbf{1}_{m\left\vert \mathbf{g}\right\vert ^{2}>4\Delta I}d%
\boldsymbol{\xi }_{\ast }dI_{\ast }dI^{\prime }dI_{\ast }^{\prime }\text{,}
\end{eqnarray*}%
with 
\begin{eqnarray*}
\widehat{\delta }_{1} &=&\mathbf{1}_{m\left\vert \mathbf{g}\right\vert
^{2}>4\Delta I}\delta _{1}\left( \sqrt{\left\vert \mathbf{g}\right\vert ^{2}-%
\frac{4}{m}\Delta I}-\left\vert \mathbf{g}^{\prime }\right\vert \right) \\
&=&\mathbf{1}_{m\left\vert \mathbf{g}\right\vert ^{2}>4\Delta I}\delta
_{1}\left( \sqrt{E-\left( I^{\prime }+I_{\ast }^{\prime }\right) }%
-\left\vert \mathbf{g}^{\prime }\right\vert \right) .
\end{eqnarray*}%
Then%
\begin{eqnarray*}
\upsilon &\geq &C\int_{\mathbb{R}^{3}\times \mathbb{R}_{+}}\int_{I^{\prime
}+I_{\ast }^{\prime }\leq E/2}e^{-I_{\ast }}e^{-m\left\vert \boldsymbol{\xi }%
_{\ast }\right\vert ^{2}/2}\dfrac{\sqrt{E-\left( I^{\prime }+I_{\ast
}^{\prime }\right) }}{E^{\delta +\left( \alpha -1\right) /2}}\left( I_{\ast
}I^{\prime }I_{\ast }^{\prime }\right) ^{\delta /2-1}\, \\
&&d\boldsymbol{\xi }_{\ast }dI_{\ast }dI^{\prime }dI_{\ast }^{\prime } \\
&\geq &C\int_{\mathbb{R}^{3}\times \mathbb{R}_{+}}\dfrac{E^{1/2}e^{-m\left%
\vert \boldsymbol{\xi }_{\ast }\right\vert ^{2}/2}}{E^{\delta +\left( \alpha
-1\right) /2}}I_{\ast }^{\delta /2-1}e^{-I_{\ast }}\left(
\int_{0}^{E/4}\left( I^{\prime }\right) ^{\delta /2-1}dI^{\prime }\right)
^{2}\,d\boldsymbol{\xi }_{\ast }dI_{\ast } \\
&=&C\int_{\mathbb{R}^{3}\times \mathbb{R}_{+}}E^{1-\alpha /2}e^{-m\left\vert 
\boldsymbol{\xi }_{\ast }\right\vert ^{2}/2}I_{\ast }^{\delta
/2-1}e^{-I_{\ast }}\,d\boldsymbol{\xi }_{\ast }dI_{\ast } \\
&\geq &C\int_{\mathbb{R}^{3}}\left( \left\vert \mathbf{g}\right\vert
^{2}+I\right) ^{1-\alpha /2}e^{-m\left\vert \boldsymbol{\xi }_{\ast
}\right\vert ^{2}/2}\,d\boldsymbol{\xi }_{\ast }\int_{0}^{\infty }I_{\ast
}^{\delta /2-1}e^{-I_{\ast }}\,dI_{\ast } \\
&\geq &C\int_{\mathbb{R}^{3}}\left( \left\vert \left\vert \boldsymbol{\xi }%
\right\vert -\left\vert \boldsymbol{\xi }_{\ast }\right\vert \right\vert
^{2}+I\right) ^{1-\alpha /2}e^{-m\left\vert \boldsymbol{\xi }_{\ast
}\right\vert ^{2}/2}\,d\boldsymbol{\xi }_{\ast }\text{.}
\end{eqnarray*}

Consider two different cases separately: $\left\vert \boldsymbol{\xi }%
\right\vert \leq 1$ and $\left\vert \boldsymbol{\xi }\right\vert \geq 1$.
Firstly, if $\left\vert \boldsymbol{\xi }\right\vert \geq 1$, then%
\begin{eqnarray*}
\upsilon &\geq &C\int_{\left\vert \boldsymbol{\xi }_{\ast }\right\vert \leq
1/2}\left( \left( \left\vert \boldsymbol{\xi }\right\vert -\left\vert 
\boldsymbol{\xi }_{\ast }\right\vert \right) ^{2}+I\right) ^{1-\alpha
/2}e^{-m\left\vert \boldsymbol{\xi }_{\ast }\right\vert ^{2}/2}\,d%
\boldsymbol{\xi }_{\ast } \\
&\geq &C\left( \left\vert \boldsymbol{\xi }\right\vert ^{2}+I\right)
^{1-\alpha /2}\int_{\left\vert \boldsymbol{\xi }_{\ast }\right\vert \leq
1/2}e^{-m/8}\,d\boldsymbol{\xi }_{\ast }=C\left( \left\vert \boldsymbol{\xi }%
\right\vert ^{2}+I\right) ^{1-\alpha /2} \\
&\geq &C\left( \left\vert \boldsymbol{\xi }\right\vert +\sqrt{I}\right)
^{2-\alpha }\geq C\left( 1+\left\vert \boldsymbol{\xi }\right\vert +\sqrt{I}%
\right) ^{2-\alpha }\text{.}
\end{eqnarray*}%
Secondly, if $\left\vert \boldsymbol{\xi }\right\vert \leq 1$, then%
\begin{eqnarray*}
\upsilon &\geq &C\int_{\left\vert \boldsymbol{\xi }_{\ast }\right\vert \geq
2}\left( \left( \left\vert \boldsymbol{\xi }_{\ast }\right\vert -\left\vert 
\boldsymbol{\xi }\right\vert \right) ^{2}+I\right) ^{1-\alpha
/2}e^{-m\left\vert \boldsymbol{\xi }_{\ast }\right\vert ^{2}/2}\,d%
\boldsymbol{\xi }_{\ast } \\
&\geq &C\left( 1+I\right) ^{1-\alpha /2}\int_{\left\vert \boldsymbol{\xi }%
_{\ast }\right\vert \geq 2}e^{-m\left\vert \boldsymbol{\xi }_{\ast
}\right\vert ^{2}/2}\,d\boldsymbol{\xi }_{\ast }=C(1+I)^{1-\alpha /2} \\
&\geq &C\left( 1+\sqrt{I}\right) ^{2-\alpha }\geq C\left( 1+\left\vert 
\boldsymbol{\xi }\right\vert +\sqrt{I}\right) ^{2-\alpha }\text{.}
\end{eqnarray*}%
Hence, there is a positive constant $\upsilon _{-}>0$, such that $\upsilon
\geq \upsilon _{-}\left( 1+\left\vert \boldsymbol{\xi }\right\vert +\sqrt{I}%
\right) ^{2-\alpha }$ for all $\boldsymbol{\xi }\in \mathbb{R}^{3}$.

On the other hand, for any positive number $\varepsilon >0$%
\begin{eqnarray*}
\upsilon &\leq &C\int_{\mathbb{R}^{3}\times \left( \mathbb{R}_{+}\right)
^{3}}e^{-I_{\ast }}e^{-m\left\vert \boldsymbol{\xi }_{\ast }\right\vert
^{2}/2}\dfrac{\left( I_{\ast }I^{\prime }I_{\ast }^{\prime }\right) ^{\delta
/2-1}}{E^{\delta +\alpha /2-1}}\,d\boldsymbol{\xi }_{\ast }dI_{\ast
}dI^{\prime }dI_{\ast }^{\prime } \\
&=&C\int_{\mathbb{R}^{3}}e^{-m\left\vert \boldsymbol{\xi }_{\ast
}\right\vert ^{2}/2}\int_{0}^{\infty }I_{\ast }^{\delta /2-1}e^{-I_{\ast
}}\left( \int_{0}^{1}\dfrac{\left( I^{\prime }\right) ^{\delta /2-1}}{%
E^{\delta /2+\alpha /4-1/2}}\,dI^{\prime }\right) ^{2}dI_{\ast }d\boldsymbol{%
\xi }_{\ast } \\
&&+2C\int_{\mathbb{R}^{3}}e^{-m\left\vert \boldsymbol{\xi }_{\ast
}\right\vert ^{2}/2}\int_{0}^{\infty }I_{\ast }^{\delta /2-1}e^{-I_{\ast
}}\int_{1}^{\infty }\dfrac{\left( I_{\ast }^{\prime }\right) ^{\delta /2-1}}{%
E^{\left( \delta +\alpha \right) /2+1/4}}dI_{\ast }^{\prime } \\
&&\times \int_{0}^{1}\dfrac{\left( I^{\prime }\right) ^{\delta /2-1}}{%
E^{\delta /2-5/4}}\,dI^{\prime }dI_{\ast }d\boldsymbol{\xi }_{\ast } \\
&&+C\int_{\mathbb{R}^{3}\times \mathbb{R}_{+}}\!\!\!E^{1+\left( \varepsilon
-\alpha \right) /2}e^{-m\left\vert \boldsymbol{\xi }_{\ast }\right\vert
^{2}/2}I_{\ast }^{\delta /2-1}e^{-I_{\ast }}\left( \int_{1}^{\infty }\dfrac{%
\left( I^{\prime }\right) ^{\delta /2-1}}{E^{\delta /2+\varepsilon /4}}%
dI^{\prime }\right) ^{2}\!d\boldsymbol{\xi }_{\ast }dI_{\ast } \\
&\leq &C\int_{\mathbb{R}^{3}}e^{-m\left\vert \boldsymbol{\xi }_{\ast
}\right\vert ^{2}/2}d\boldsymbol{\xi }_{\ast }\int_{0}^{\infty }I_{\ast
}^{\delta /2-1}e^{-I_{\ast }}dI_{\ast }\left( \int_{0}^{1}\dfrac{1}{\left(
I^{\prime }\right) ^{\alpha /4+1/2}}\,dI^{\prime }\right) ^{2} \\
&&+C\!\int_{\mathbb{R}^{3}}\!e^{-m\left\vert \boldsymbol{\xi }_{\ast
}\right\vert ^{2}/2}d\boldsymbol{\xi }_{\ast }\int_{0}^{\infty }\!\!I_{\ast
}^{\delta /2-1}e^{-I_{\ast }}dI_{\ast }\int_{1}^{\infty }\!\!\dfrac{1}{%
\left( I_{\ast }^{\prime }\right) ^{\alpha /2+5/4}}dI_{\ast }^{\prime
}\int_{0}^{1}\!\!\left( I^{\prime }\right) ^{1/4}dI^{\prime } \\
&&+C\int_{\mathbb{R}^{3}\times \mathbb{R}_{+}}\!\!\!E^{1+\left( \varepsilon
-\alpha \right) /2}e^{-m\left\vert \boldsymbol{\xi }_{\ast }\right\vert
^{2}/2}I_{\ast }^{\delta /2-1}e^{-I_{\ast }}d\boldsymbol{\xi }_{\ast
}dI_{\ast }\!\left( \int_{1}^{\infty }\dfrac{1}{\left( I_{\ast }^{\prime
}\right) ^{1+\varepsilon /4}}dI_{\ast }^{\prime }\right) ^{2} \\
&\leq &C\left( 1+\int_{\mathbb{R}^{3}}\left( 1+m\frac{\left\vert \mathbf{g}%
\right\vert ^{2}}{4}+I\right) ^{1+\left( \varepsilon -\alpha \right)
/2}e^{-m\left\vert \boldsymbol{\xi }_{\ast }\right\vert ^{2}/2}d\boldsymbol{%
\xi }_{\ast }\right. \\
&&\times \left. \int_{0}^{\infty }\left( 1+I_{\ast }\right) I_{\ast
}^{\delta /2-1}e^{-I_{\ast }}dI_{\ast }\right) \\
&\leq &C\left( 1+\left( 1+\left\vert \boldsymbol{\xi }\right\vert
^{2}+I\right) ^{1+\left( \varepsilon -\alpha \right) /2}\int_{0}^{\infty
}\left( 1+r^{2}\right) ^{1+\varepsilon }r^{2}e^{-2r^{2}}dr\right) \\
&\leq &C\left( 1+\left\vert \boldsymbol{\xi }\right\vert +\sqrt{I}\right)
^{2-\alpha +\varepsilon }\text{.}
\end{eqnarray*}%
Hence, there is a positive constant $\upsilon _{+}>0$, such that $\upsilon
\leq \upsilon _{+}\!\!\left( 1+\left\vert \boldsymbol{\xi }\right\vert +%
\sqrt{I}\right) ^{2-\alpha +\varepsilon }$ for all $\boldsymbol{\xi }\in 
\mathbb{R}^{3}$.
\end{proof}

\bibliographystyle{AIMS}
\bibliography{biblo1}

\end{document}